\documentclass{amsart} \usepackage{url}
\usepackage{amsmath,amsfonts,euscript,amsthm,amssymb,amscd}
\usepackage[T1]{fontenc}
\usepackage{lmodern}
\usepackage{epigraph}
\usepackage[utf8]{inputenc}
\usepackage{enumerate}
\usepackage{tikz}

\makeatletter
\renewcommand{\tocsection}[3]{%
  \indentlabel{\@ifnotempty{#2}{\bfseries\ignorespaces#1 #2\quad}}\bfseries#3}
\renewcommand{\tocsubsection}[3]{%
  \indentlabel{\@ifnotempty{#2}{\ignorespaces#1 #2\quad}}#3}

\newcommand\@dotsep{4.5}
\def\@tocline#1#2#3#4#5#6#7{\relax
  \ifnum #1>\c@tocdepth 
  \else
    \par \addpenalty\@secpenalty\addvspace{#2}%
    \begingroup \hyphenpenalty\@M
    \@ifempty{#4}{%
      \@tempdima\csname r@tocindent\number#1\endcsname\relax
    }{%
      \@tempdima#4\relax
    }%
    \parindent\z@ \leftskip#3\relax \advance\leftskip\@tempdima\relax
    \rightskip\@pnumwidth plus1em \parfillskip-\@pnumwidth
    #5\leavevmode\hskip-\@tempdima{#6}\nobreak
    \leaders\hbox{$\m@th\mkern \@dotsep mu\hbox{.}\mkern \@dotsep mu$}\hfill
    \nobreak
    \hbox to\@pnumwidth{\@tocpagenum{\ifnum#1=1\bfseries\fi#7}}\par
    \nobreak
    \endgroup
  \fi}
\AtBeginDocument{%
\expandafter\renewcommand\csname r@tocindent0\endcsname{0pt}
}
\def\l@subsection{\@tocline{2}{0pt}{2.5pc}{5pc}{}}
\makeatother

\usepackage{lipsum}

\theoremstyle{definition}
\newtheorem{Def}{Definition}[subsection]
\newtheorem{Def-Prop}[Def]{Definition-Proposition}

\newtheorem{Th}[Def]{Theorem}

\newtheorem{remark}[Def]{Remark}
\newtheorem{Prop}[Def]{Proposition}
\newtheorem{Lemma}[Def]{Lemma}

\newtheorem{Th1}{Theorem}

\DeclareMathOperator*{\summm}{\oplus}

\DeclareMathOperator*{\cdim}{codim}
\newcommand{\hdot}{{\:\raisebox{3pt}{\text{\circle*{1.5}}}}}
\newcommand{\hdotc}{{\:\raisebox{1pt}{\text{\circle*{1.5}}}}}

 \newcommand{\Ba}{\begin{array}}
 \newcommand{\Ea}{\end{array}}
\usepackage[all]{xy}
\usepackage{tikz}
 \tikzstyle{int}=[circle, draw,fill=black,outer sep=0,minimum size=3pt, inner sep=0]
  \tikzstyle{ext}=[circle, draw=black,outer sep=0,inner sep=1pt]
\input{diag}
\tikzset{snakeit/.style={decorate, decoration={snake, amplitude=.2mm,segment length=1mm}}}

\tikzset{ext/.style={circle, draw,inner sep=1pt}, int/.style={circle,draw,fill,inner sep=2pt},nil/.style={inner sep=1pt}}
\tikzset{cy/.style={circle,draw,fill,inner sep=2pt},scy/.style={circle,draw,inner sep=2pt},scyx/.style={draw,cross out,inner sep=2pt},scyt/.style={draw,regular polygon,regular polygon sides=3,inner sep=0.95pt}}
\tikzset{exte/.style={circle, draw,inner sep=3pt},inte/.style={circle,draw,fill,inner sep=3pt}}
\tikzset{diagram/.style={matrix of math nodes, row sep=3em, column sep=2.5em, text height=1.5ex, text depth=0.25ex}}
\tikzset{diagram2/.style={matrix of math nodes, row sep=0.5em, column sep=0.5em, text height=1.5ex, text depth=0.25ex}}
\tikzset{rowcolsep/.style={column sep=.2cm, row sep=.1cm}}

\tikzset{
  crossed/.style={
    decoration={markings,mark=at position .5 with {\arrow{|}}},
    postaction={decorate},
    shorten >=0.4pt}}

\tikzset{every picture/.style={baseline=-.65ex} }

\begin{document}

\author{Alexey Kalugin} 
\address{Max Planck Institut für Mathematik in den Naturwissenschaften, Inselstraße 22, 04103 Leipzig, Germany}
\email{alexey.kalugin@mis.mpg.de}
\title{On the Caldararu and Willwacher conjectures}
\maketitle
\begin{abstract} In the present paper, we study a relation between the cohomology of moduli stacks of smooth and proper curves $\mathcal M_{g,n}$ and the cohomology of ribbon graph complexes. The main results of this work are proofs of T. Willwacher's conjecture and A. C$\check{\mathrm a}$ld$\check{\mathrm a}$raru's conjecture about the cohomology of the Bridgeland differential. We also discuss the relation to string topology and the Chan-Galatius-Payne theorem about the weight zero part of the compactly supported cohomology of $\mathcal M_{g,n}.$ 
\end{abstract} \tableofcontents

\section*{Introduction} 

\subsection{Introduction} In this paper, we study the relationship between the cohomology of moduli stacks of smooth proper algebraic curves with marked points $\mathcal M_{g,n}$ and the cohomology of ribbon graph complexes. For an integer $d\in \mathbb Z$ we denote by $\textsf {RGC}_d^{\hdot}(\delta)$ the \textit{Kontsevich-Penner ribbon graph complex} \cite{Kon1} \cite{Kon2} \cite{Kon3} \cite{Pen} \cite{MW}. This is a combinatorial cochain complex with entries being ribbon graphs (also called fat graphs) i.e. graphs equipped with cyclic order on a set of the half-edges incident to each vertex and a choice of a certain orientation (for details, see \cite{MW}). The differential acts by splitting a vertex in all possible ways which preserve a cyclic order.
\par\medskip 
 $$\Gamma_1={\xy
(0,-2)*{\overset{\bullet}{}}="A";
(0,-2)*{\overset{\bullet}{}}="B";
"A"; "B" **\crv{(6,6) & (-6,6)};
\endxy},
 \qquad \Gamma_2=\xy
(0,-2)*{\bullet}="A";
(0,-2)*{\bullet}="B";
"A"; "B" **\crv{(-4,6) & (10,6)};
"A"; "B" **\crv{(-10,6) & (4,6)};
\endxy
,\qquad 
\Gamma_3=
\hspace{-3mm}\Ba{c}\mbox{\xy
(0,-2)*{\bullet}="A";
(0,-2)*{\bullet}="B";
"A"; "B" **\crv{(6,6) & (-6,6)};
"A"; "B" **\crv{(6,-10) & (-6,-10)};
\endxy}\Ea \hspace{-3mm}
$$
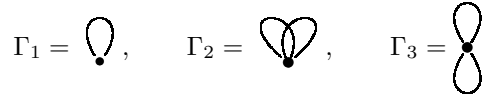
{\centering Examples of ribbon graphs.
\par\medskip }
According to the fundamental results of D. Mumford, R. C. Penner, J. Harer, and W. Thurston \cite{Mum}  \cite{Pen} \cite{Harer} there is a quasi-isomorphism of complexes:\footnote{Here $\epsilon_n$ is a one-dimensional sign local system.}
\begin{equation}\label{fund1}
\textsf {RGC}_d^{\hdot}(\delta)_{\geq 3}\cong \prod_{g\geq 0, n\geq 1, 2g+n-2>0}^{\infty} C_c^{\hdot+2dg-n}(\mathcal M_{g,n}/\Sigma_n,\epsilon_n).
\end{equation}
This result comes from considering a topological stack of \textit{metric} ribbon graphs $M_{g,n}^{rib}$ and constructing a homeomorphism $M_{g,n}^{rib}\cong \mathcal M_{g,n}\times \mathbb R^n_{>0}.$
The description of the cohomology of moduli stacks of curves in ribbon graphs plays a crucial role in the study of $\mathcal M_{g,n}.$ In particular, J. Harer and D. Zagier \cite{HZ} calculated the Euler characteristic of moduli stacks of curves and M. Kontsevich \cite{Kon3} proved E. Witten's conjecture. 

\subsection{Main results} The first main result of this paper is a proof of the following Theorem conjectured by T. Willwacher \cite{WillL}:
\begin{Th1}[T. Willwacher's conjecture] For every $g\geq0$ and $n\geq 1$ such that $2g+n-2>0,$ the following square commutes (in the derived category of vector spaces):

\begin{equation*}
\begin{diagram}[height=3.1em,width=3.4em]
\prod_{g,n}^{\infty} C_c^{\hdot+2dg-n}(\mathcal M_{g,n}/ \Sigma_{n},\epsilon_n) & &  \rTo_{}^{\nabla_1} &  &   \prod_{g,n}^{\infty}C_{c}^{\hdot+2dg-n}(\mathcal M_{g,n+1}/ \Sigma_{n+1},\epsilon_{n+1}) &  \\
\dTo_{\sim}^{} & & &  & \dTo_ {\sim}^{}  && \\
\textsf{RGC}_{d}^{\hdot}(\delta)_{\geq 3} & &  \rTo_{}^{\Delta_1} &  & \textsf{RGC}_{d}^{\hdot+1}(\delta)_{\geq 3}   \\
\end{diagram}
\end{equation*}
\end{Th1}
Where $\nabla_1$ is the so-called \textit{Willwacher differential} introduced in \cite{AK} which is defined\footnote{This morphism exists at the level of \textit{derived categories} due to non-properness of the forgetful morphism.} as a pullback along the forgetful morphism:
$$
\pi\colon \mathcal M_{g,n+1}\longrightarrow \mathcal M_{g,n},
$$
and $\Delta_1$ is the so-called \textit{Bridgeland differential} introduced by T. Bridgeland at the beginning of $00'$s in his studies of mirror symmetry \cite{Brid}. The differential $\Delta_1$ is defined by adding an edge to a ribbon graph, splitting a boundary in two:\footnote{More precisely, we connect \textit{different corners} attached to the same boundary i.e. intervals on "blown-up" vertices of a ribbon graph lying between half-edges \cite{MW}.} 
$$
\Delta_1\colon 
{\xy
(5,5)*{\overset{\bullet}{}}="A";
(2,5)*{\overset{\bullet}{}}="B";
(-1,1)*{\overset{\bullet}{}}="C";
(8,1)*{\overset{\bullet}{}}="D";
(2,-3)*{\overset{\bullet}{}}="F";
(5,-3)*{\overset{\bullet}{}}="E";
\ar @{-} "B";"C" <0pt>
\ar @{-} "A";"B" <0pt>
\ar @{-} "A";"D" <0pt>
\ar @{-} "C";"F" <0pt>
\ar @{-} "F";"E" <0pt>
\ar @{-} "E";"D"
 <0pt>\endxy}
\longmapsto \sum 
{\xy
(5,5)*{\overset{\bullet}{}}="A";
(2,5)*{\overset{\bullet}{}}="B";
(-1,1)*{\overset{\bullet}{}}="C";
(8,1)*{\overset{\bullet}{}}="D";
(2,-3)*{\overset{\bullet}{}}="F";
(5,-3)*{\overset{\bullet}{}}="E";
\ar @{-} "B";"C" <0pt>
\ar @{-} "A";"B" <0pt>
\ar @{-} "A";"D" <0pt>
\ar @{-} "C";"F" <0pt>
\ar @{-} "F";"E" <0pt>
\ar @{-} "E";"D"<0pt>
\ar @{-} "F";"A"<0pt>
\endxy}
$$
Note that the Bridgeland differential has a combinatorial origin and naturally comes as the "dual" of the standard differential $\delta$ i.e. there is an involution on ribbon graphs defined by interchanging vertices with boundaries \cite{CFL}, under this duality $\Delta_1$ goes to $\delta.$ \textit{Per contra} the Willwacher differential has "motivic" origins i.e. $\nabla_1$ preserves associated graded quotients of the weight filtration on $H_c^{\hdot}(\mathcal M_{g,n}/ \Sigma_n,\epsilon_{n})$ \cite{AK}. 
\par\medskip 
Recall that the \textit{Merkulov-Willwacher ribbon graph complex} $\textsf {RGC}_d^{\hdot}(\delta+\Delta_1)$ is a version of the Kontsevich-Penner ribbon graph complex where the differential is twisted by the Bridgeland differential \cite{MW}. In \textit{ibid.} the \textit{properad of ribbon graphs} $\textsf{RGra}_d$ was introduced as collections of ribbon graphs together with a certain composition. Remarkably this properad is equipped with a morphism from a properad $\Lambda \textsf{LieB}_{d,d}$ which controls Lie bialgebras in the sense of V. Drinfeld \cite{Drin}. Using the theory of deformation complexes of properads we set:
$$
\textsf {RGC}_d^{\hdot}(\delta+\Delta_1):=\mathrm {Def}(\Lambda \textsf{LieB}_{d,d}\longrightarrow \textsf{RGra}_d).
$$
Applying Theorem $1$ and Theorem $4.4.5$ from \cite{AK} we compute the cohomology of the Merkulov-Willwacher ribbon graph complex and obtain the proof of the following conjecture \cite{Cal1} \cite{AWZ}: 
\begin{Th1}[A. C$\check{\mathrm a}$ld$\check{\mathrm a}$raru's conjecture] For every $d\in \mathbb Z$ the Merkulov-Willwacher ribbon graph complex $\textsf{RGC}_d^{\hdot}(\delta+\Delta_1)$ is quasi-isomorphic to the totality of the shifted cohomology with compact support of moduli stacks $\mathcal M_g:$\footnote{Here by a sign $\cong_{\geq 2}$ we indicate a quasi-isomorphism in genus $g\geq 2,$ however we also compute the cohomology of the Merkulov-Willwacher complex in genus one and zero (Proposition \ref{CalC1}).} 
$$
\textsf{RGC}_d^{\hdot}(\delta+\Delta_1)\cong_{\geq 2} \prod_{g\geq 2}^{\infty} C^{\hdot+2dg-1}_c(\mathcal M_g,\mathbb Q).
$$
\end{Th1}
Put another way A. C$\check{\mathrm a}$ld$\check{\mathrm a}$raru's conjecture extends quasi-isomorphism \eqref{fund1} to the non-punctured case by introducing the combinatorial model for the compactly supported cohomology of $\mathcal M_g$ for $g\geq 2.$ 
\par\medskip 

Denote by $\textsf {GC}_d^2$ the \textit{Kontsevich graph complex} \cite{Kon1}\cite{Kon2}\cite{Will} i.e. the combinatorial cochain complex with entries being at least two-valent graphs with an orientation (see \cite{Will} for the precise definitions) and a differential acts by splitting a vertex in all possible ways. 
\begin{Th1} For every $d\in \mathbb Z$ we have the canonical injection:
$$
H^{\hdot}(\textsf {GC}_{2d}^2)\longrightarrow H^{\hdot+1}(\textsf{RGC}_d(\delta+\Delta_1))
$$
\end{Th1}
This result immediately follows from the M. Chan S. Galatius and S. Payne's description of the weight zero part of the compactly supported cohomology of $\mathcal M_g$ \cite{CGP1} (see also \cite{AWZ},\cite{AK}) together with Theorem $2$ and explicit computations in genus one. 

\subsection{Techniques} The main idea in the proof of Theorem $1$ is straightforward: we show that both differentials $\Delta_1$ and $\nabla_1$ are induced by the pullback along the morphism on moduli spaces of bordered surfaces which forgets a "complex marking". A crucial role in this construction is played by moduli spaces of stable bordered surfaces in the sense of M. Liu and K. Costello \cite{Liu} \cite{Cost}. 
\par\medskip 
Let $g\geq 0$ and $n\geq 1$ such that $2g+n-2>0,$ recall that \textit{Costello's moduli space} $\overline{\mathcal N}_{g,n}$ is a moduli space of stable bordered surfaces of genus $g$ with $n$-boundary components such that we only allow nodes on the boundary. This moduli space is a non-compact but smooth orbifold with corners. Analogous to the Deligne-Mumford moduli stacks, these moduli spaces have versions with marked points. In the world of bordered surfaces, we have different versions of marked points depending on where we allow them: on a boundary or (and) in an interior of the surface. Following \textit{ibid.} for $4g+2n+p+2q-4>0,$ we have a moduli space $\overline{\mathcal N}_{g,n,p,q}$ of stable bordered surfaces with singularities only on a boundary and $p$ points on a boundary (real points) and $q$ points in an interior of the surface (complex points). We have morphisms:
$$
\pi^{real}_{p}\colon \overline{\mathcal N}_{g,n,p+1,q}\longrightarrow \overline{\mathcal N}_{g,n,p,q},\quad \pi^{comp}_q\colon \overline{\mathcal N}_{g,n,p,q+1}\longrightarrow \overline{\mathcal N}_{g,n,p,q},
$$
which forget real or complex marking and stabilise the resulting bordered surface. Following \textit{ibid.} we have a diagram of weak homotopy equivalences (the \textit{Costello homotopy equivalence}):
$$
v\colon \mathcal N_{g,n} \longrightarrow \overline{\mathcal N}_{g,n}\longleftarrow D_{g,n}\colon u.
$$
Where $\mathcal N_{g,n}$ is a locus of smooth bordered surfaces i.e. an interior of $\overline{\mathcal N}_{g,n}$ and $D_{g,n}$ is a locus of bordered surfaces with irreducible components of genus $0$ (the maximally degenerated locus). Note that analogous to the Deligne-Mumford moduli stacks $\overline{\mathcal M}_{g,n}$ the topological stack $D_{g,n}$ can be stratified by dual graphs of surfaces. A natural orientation of disks produces a cyclic order on connected components of strata. Hence applying the Cousin resolution one can show that the chain complex which computes the homology of $D_{g,n}$ with coefficients in the sign local system can be identified with a direct summand of the Kontsevich-Penner ribbon graph complex (see \cite{Cost} and Proposition \ref{expr}). 
\par\medskip 
The Costello moduli space $\overline{\mathcal N}_{g,n}$ is an open suborbifold in the \textit{Liu moduli space} $\overline{\mathcal N}_{g,n}^L$ \cite{Liu}. The latter moduli space consists of stable bordered surfaces of genus $g$ with $n$ (non-labelled) boundary components where we allow not only nodes on the boundary of the surface but also nodes in the interior and shrinking lengths of boundaries to zero. The moduli space $\overline{\mathcal N}_{g,n}^L$ is a smooth and compact orbifold with corners of dimension $6g-6+3n.$ Our first result is the following:

\begin{Th1}\label{th4} We have a diagram of weak homotopy equivalences of topological stacks:\footnote{This result is established with the help of $CAT(0)$-properties of the Weil-Petersson geometry of the moduli space $\overline{\mathcal N}_{g,n}^L.$} 
$$
a\colon \mathcal M_{g,n}/ \Sigma_n\longrightarrow \widetilde{\mathcal N}_{g,n}\longleftarrow \mathcal N_{g,n}\colon b.
$$
\end{Th1}
Where $\widetilde{\mathcal N}_{g,n}$ is an open smooth substack of $\overline{\mathcal N}_{g,n}^L,$ which consists of stable bordered surfaces with singularities given by shrinking lengths of boundaries to zero (cusps). And  $\mathcal M_{g,n}/ \Sigma_n$ is identified (topologically) with the maximal degenerated locus in $\widetilde{\mathcal N}_{g,n}^L$ by replacing cusps with marked points. Theorem $\ref{th4}$ allows us to define the so-called \textit{boundary shrinking morphism}:\footnote{We believe that this construction should be well known to experts, however, we were unable to find a reference (see \cite{AGOL}).} 
$$
\rho_{\mathcal M_{g,n}/\Sigma_n}:=a_!\circ b_!^{-1}\colon C_c^{\hdot}(\mathcal N_{g,n},\mathbb Q) \overset{\sim}{\longrightarrow} C_c^{\hdot-n}(\mathcal M_{g,n}/\Sigma_n,\epsilon_n).$$
Note that the morphism $\rho_{\mathcal M_{g,n}/\Sigma_n}$ exists at the level of complexes and can be identified with the Gysin pushforward along the "topological" boundary shrinking morphism, which is defined as the Weil-Petersson closest point (see Remark \ref{topbsh}).

Building on K. Costello's homotopy equivalence, we explicitly relate the homology of $D_{g,n}$ to the compactly supported cohomology of $\mathcal M_{g,n}/\Sigma_n:$ 
\begin{align*}
\textsf {RGC}_d^{\hdot}(\delta)\overset{u_*}{\longrightarrow} \prod_{g,n}^{\infty}  C_{6g-6+3n-2dg-\hdotc}&( \overline{\mathcal N}_{g,n},\epsilon_n)\\
\overset{v_*}{\overset{\sim}{\longleftarrow}}  \prod_{g,n}^{\infty}  C_{6g-6+3n-2dg-\hdotc}({\mathcal N}_{g,n}&,\epsilon_n)\overset{\mathbb {D}}{\overset{\sim}{\longleftarrow}}  \prod_{g,n}^{\infty}  C^{\hdot+2dg}_c({\mathcal N}_{g,n},\mathbb Q)\\
\overset{\rho_{\mathcal M_{g,n}/\Sigma_n}}{\overset{\sim}{\longrightarrow}}   &\prod_{g,n}^{\infty} C_c^{\hdot+2dg-n}(\mathcal M_{g,n}/\Sigma_n,\epsilon_n)
\end{align*}
This result gives an alternative proof of fundamental quasi-isomorphism \eqref{fund1} and has an independent interest (see \cite{Mer}) . 
\par\medskip 
We define the \textit{geometric Bridgeland differential} $\Delta_1^{bor}$ as the following composition: we take a Gysin pullback along the proper morphism, which forgets the marking in the interior (induced by $\pi^{comp}$):
$$
\pi_{D_{g,n,0,1}}\colon D_{g,n,0,1}\longrightarrow D_{g,n},
$$
and then we take a pushforward along the inclusion to $K_{g,n+1}^L.$ The latter space is defined as the closure (in Liu's moduli space) of the locus of nodal surfaces of the following type:
\begin{enumerate}[(a)]
\item All irreducible components are disks;
\item All irreducible components are disks except the unique one, which is an annulus.
\end{enumerate} 
\par\medskip 
The canonical inclusion $D_{g,n+1}\rightarrow K_{g,n+1}^L$ is a weak-homotopy equivalence (Lemma \ref{B0}) and we have a morphism $\xi_{D_{g,n,0,1}}\colon D_{g,n,0,1}\rightarrow K_{g,n+1}^L,$ defined by replacing a marked point with a cusp and hence we define the geometric Bridgeland differential $\Delta_1^{bor}$ as the following zigzag morphism:
\begin{equation*}\label{}
\begin{diagram}[height=2.7em,width=4.3em]
&& C_{\hdotc+2}(D_{g,n,0,1},\xi_{D_{g,n,0,1}}^*\epsilon_{n+1})   && \\
& \ldTo^{\sim} & & \rdTo^{\xi_{D_{g,n,0,1}*}} & \\ 
 C_{\hdotc}(D_{g,n,0,1},\pi^!_{D_{g,n,0,1}}\epsilon_n)   &  &    & & C_{\hdotc+2}(K_{g,n+1}^L,\epsilon_{n+1}) &  \\
\uTo_{\pi^!_{D_{g,n,0,1}}}^{}  & &  &  &  \uTo_ {\sim}^{n_*}  && \\
 C_{\hdotc}(D_{g,n},\epsilon_n) &    & \rDotsto^{\Delta_1^{bor}} &   &  C_{\hdotc+2}(D_{g,n+1},\epsilon_{n+1})    \\
\end{diagram}
\end{equation*}
To relate the morphism $\Delta_1^{bor}$ to the Bridgeland differential $\Delta_1$ we have to introduce the intermediate operator $\Delta_1^{bor'}.$ Denote by $D_{g,n,2}'\subset D_{g,n,0,2}$ the locus which consists of stable bordered surfaces where two (non-labelled) markings belong to the same boundary. We have morphisms:
$$
\pi_{D_{g,n,2}'}\colon D_{g,n,2}'\longrightarrow D_{g,n},\qquad \xi_{D_{g,n,2}'}\colon D_{g,n,2}'\longrightarrow D_{g,n+1}
$$
This first morphism is induced by the morphism $\pi^{real}$ and hence is given by forgetting marked points and stabilising the resulting nodal curve. The second morphism is a version of a clutching morphism for bordered surfaces i.e. we connect two marked points by a node (cf. \cite{KN}). We set:
\begin{equation*}\label{}
\begin{diagram}[height=2.7em,width=4.2em]
 C_{\hdotc}^{}(D_{g,n,2}',\pi_{D_{g,n,2}'}^!\epsilon_{n})  &  & \lTo_{\sim}^{}   & & C_{\hdotc+2}^{}(D_{g,n,2}',\xi_{D_{g,n,2}'}^*\epsilon_{n+1}) &  \\
\uTo_{\pi^!_{D_{g,n,2}'}}^{}  & &  &  &  \dTo_{}^{\xi_{D_{g,n,2}'*}}  && \\
 C_{\hdotc}(D_{g,n},\epsilon_n) &    & \rDotsto^{\Delta^{bor'}_1} &   &  C_{\hdotc+2}(D_{g,n+1},\epsilon_{n+1})    \\
\end{diagram}
\end{equation*}
Applying the explicit computations of the Gysin pullback we prove (Proposition \ref{KeyL4} and Proposition \ref{KeyL3}) that we have the following decomposition:\footnote{The reason for this is that a possible degeneration of a surface having an annulus as an irreducible component is given by contracting arcs with endpoints on two boundaries of an annulus, this leads to the morphism which connects \textit{identical} corners by an edge (see \cite{DHV}).}
$$
\Delta_1^{bor'}=\Delta_1^{bor}+\Delta_1
$$
Finally, we show (Proposition \ref{KeyL3}) that the morphism $\Delta_1^{bor'}$ vanishes, and hence the geometric Bridgeland differential coincides with $\Delta_1.$ Hence we get the following result which is \textit{le cœur} of the present work: 
\begin{Th1} For $g\geq 0,n\geq 1$ such that $2g+n-2>0,$ the following diagram commutes (in the derived category of vector spaces):
\begin{equation*}
\begin{diagram}[height=3.1em,width=3.3em]
 \prod_{g,n}^{\infty} C_{\hdotc+\dim \mathcal N_{g,n}-2dg}(D_{g,n},\epsilon_n) & &  \rTo_{}^{\Delta_1^{bor}} &  &   \prod_{g,n}^{\infty} C_{\hdotc+\dim \mathcal N_{g,n+1}-2dg-2}(D_{g,n+1},\epsilon_{n+1}) &  \\
\uTo_{\sim}^{} & & &  & \uTo_ {\sim}^{}  && \\
\textsf{RGC}^{\hdot}_d(\delta)_{\geq 3} & &  \rTo_{}^{\Delta_1} &  & \textsf{RGC}^{\hdot+1}_d(\delta)_{\geq 3}    \\
\end{diagram}
\end{equation*}
\end{Th1}
The geometric Bridgeland differential admits an obvious version for the moduli space of smooth bordered surfaces: we take the Gysin pullback along the (proper) morphism which forgets a marking in the interior:
$$
\pi_{\overline{\mathcal N}_{g,n,0,1}}\colon \overline{\mathcal N}_{g,n,0,1}\longrightarrow \mathcal N_{g,n}.
$$
Further, we take a push-forward to the moduli space of bordered surfaces ${\mathcal N}_{g,n+1}^L .$ The latter is a moduli space, defined as a subspace in Lui's moduli space where we do not allow singularities in the interior. The canonical inclusion $\overline{\mathcal N}_{g,n}\rightarrow {\mathcal N}_{g,n}^L$ is a homotopy equivalence. Hence we define the morphism $\nabla_1^{bor}$ as the following zig-zag morphism:
\begin{equation*}\label{}
\begin{diagram}[height=2.8em,width=3.7em]
&& C_{\hdotc+2}(\overline{\mathcal N}_{g,n,0,1},\xi^*_{\overline{\mathcal N}_{g,n,0,1}}\epsilon_{n+1})   && \\
& \ldTo_{\sim}^{} & & \rdTo^{\xi_{\overline{\mathcal N}_{g,n,0,1}*}} & \\ 
  C_{\hdotc}(\overline{\mathcal N}_{g,n,0,1},\pi^!_{\overline{\mathcal N}_{g,n,0,1}}\epsilon_n)  &  &    & & C_{\hdotc+2}({\mathcal N}_{g,n+1}^L,\epsilon_{n+1}) &  \\
\uTo_{\pi^!_{\overline{\mathcal N}_{g,n,0,1}}}^{}  & &  &  &  \uTo_ {\sim}^{q_*}  && \\
C_{\hdotc}(\overline{\mathcal N}_{g,n},\epsilon_n)  &    & &   &   C_{\hdotc+2}(\overline{\mathcal N}_{g,n+1},\epsilon_{n+1})   \\
\uTo^{v_*}_{\sim}  & &  &  &  \uTo_ {\sim}^{v_*}  && \\
C_{\hdotc}({\mathcal N}_{g,n},\epsilon_n)   &    & \rDotsto^{\nabla_1^{bor}} &   &    C_{\hdotc+2}({\mathcal N}_{g,n+1},\epsilon_{n+1})    \\\end{diagram}
\end{equation*}
The formal check (Proposition \ref{comp}) shows that under the Costello homotopy equivalence morphisms $\Delta_1^{bor}$ and $\nabla_1^{bor}$ coincide. Finally, the proof of T. Willwacher's conjecture follows from the following comparison result:

\begin{Th1} For $g\geq 0,n\geq 1$ such that $2g+n-2>0,$ the following diagram commutes (in the derived category of vector spaces):

\begin{equation*}
\begin{diagram}[height=3.1em,width=3.7em]
 C_c^{\hdot-n} (\mathcal M_{g,n}/\Sigma_{n},\epsilon_n) & & \rTo^{\nabla_1}  &  &   C_c^{\hdot-n-1} (\mathcal M_{g,n+1}/\Sigma_{n+1},\epsilon_{n+1}) &\\
 \uTo^{\rho_{\mathcal M_{g,n}/\Sigma_{n}}}_{\sim} & &  &  & \uTo_{\rho_{\mathcal M_{g,n+1}/\Sigma_{n+1}}}^{\sim}  &    \\
C_c^{\hdot}(\mathcal N_{g,n},\mathbb Q) &  & \rTo^{\mathbb D\circ \nabla_1^{bor}\circ \mathbb D} & & C_c^{\hdot+1}(\mathcal N_{g,n+1},\mathbb Q)
\end{diagram}
\end{equation*}

\end{Th1}

This Theorem compares the Willwacher differential $\nabla_1$ and the differential $\nabla^{bor}$ under the homotopy equivalence from Theorem $4.$ More precisely, we show that both differentials are induced by the forgetful morphism $\pi^{comp}\colon \overline{\mathcal N}_{g,n,0,1}^L\longrightarrow \overline{\mathcal N}_{g,n}^L$ on Liu's moduli spaces of stable bordered surfaces.  

\subsection{Relation to string topology} Our main results can be interpreted from the point of view of string topology. 
\par\medskip
Let $M$ be a smooth and oriented manifold by $M^{S^1}$ we denote the space of free loops with values in $M.$ According to M. Chas and D. Sullivan \cite{CS} $H_{\hdotc}(M^{S^1})$ carries a string product and the BV-operator coming from the circle action. If one passes to the $S^1$-equivariant homology \cite{CS1} the graded space of the relative homology $H_{\hdotc}(M^{S^1},M)$ has a shifted Lie bialgebra structure, extending results of V. Turaev \cite{Tur}. This structure has a purely algebraic (formal) counterpart, in \cite{CEL} the Lie bialgebra structure was constructed on the cyclic homology $CC_{\hdotc}(A^*[-1])$ of any Poincaré duality algebra $A$, giving the reduced part of the Chas-Sullivan structure in the case when $M$ is simply connected and closed. 

In \cite{MW} it was proved that for DG-algebra $A$ of with Poincaré duality of degree $d$ the complex $CC_{\hdotc}(A^*[-1])$  carries a natural structure of an algebra over $\textsf {RGra}_{3-d}.$ 
It is a natural question \cite{Merkul3} whenever this operation can be extended to the "non-formal" setting.\footnote{An idea of the existence of "Segal type structure" is not new and goes at least to the fundamental work \cite{CS1} (see also \cite{CG1}).} Following K. Costello we consider a bordered nodal surface $\Sigma\in D_{g,n},$ where we interpret boundaries of irreducible components of the normalisation as "inputs" and boundaries of the oriented real blow-up as "outputs":
\begin{equation*}
\begin{diagram}[height=1.9em,width=2.3em]
 \Sigma &   \lTo^{}  &   \mathrm {Blo}_{sing}\Sigma\\
\uTo & &  & \\
\overline{\Sigma}  &       \\
\end{diagram}
\end{equation*}
Applying appropriate pull-push operators, one should get operations parametrised by ribbon graphs. This structure is consistent with M. Chas and D. Sullivan Lie bialgebra structure. From this point of view, the Lie bracket and the Lie cobracket correspond to the following nodal surfaces (cf. \cite{CG1}):  
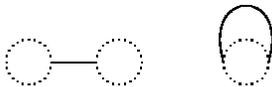
$$
\xy
 (0,0)*{
\xycircle(3,3){.}};
(12,0)*{
\xycircle(3,3){.}};
 (3,0)*{}="a",
(9,0)*{}="b",
\ar @{-} "a";"b" <0pt>
\endxy \ \ \ \ \ \ \ \
\xy
 (0,0)*{
\xycircle(3,3){.}};
(-3,0)*{}="1";
(3,0)*{}="2";
"1";"2" **\crv{(-6,10) & (6,10)};
\endxy
$$
By Theorem $2,$ every deformation of such structure is given by the compactly supported cohomology of $\mathcal M_g,$ such that the deformations of the underlying Lie bialgebra will be controlled by the weight zero quotient of $\mathcal M_g$ following Theorem $3$ (modulo Conjecture $24$ in \cite{AWZ}). Note that in \cite{AKKN} it was proved that $\textsf {GC}^{\hdot}_d$ plays a crucial role in another extension of V. Turaev's Lie bialgebra.

\subsection{Structure of the paper} In Section $2$ we recollect some facts about the (Borel-Moore) homology and operations on the (Borel-Moore) chains which we will need. In Section $3$ we recall basic constructions for ribbon graphs and definitions of the Merkulov-Willwacher ribbon graph complex and the Bridgeland differential. Section $4$ is devoted to recalling definitions of the Costello-Liu moduli spaces of stable bordered surfaces as well as Costello's result about the homotopy equivalence between $\mathcal N_{g,n}$ and $D_{g,n}.$ In this section, we prove our Theorem $4.$ Section $5$ is the heart of the present work, we make all necessary definitions and give proofs of Theorem $5$ and Theorem $6.$ Further, we get the proof of T. Willwacher's conjecture and deduce A. C$\check{\mathrm a}$ld$\check{\mathrm a}$raru's conjecture and Theorem $3.$ The last section (Appendix) is devoted to recollecting some facts about the Teichmüller space of a bordered surface and Weil-Petersson geometry of this space. 

\subsection{Acknowledgments} This work is an expanded and revised part of my Ph.D. thesis \cite{AK3}. I thank Noémie Combe, Nikita Markarian, Sergei Merkulov, Artem Prikhodko, and Sergey Shadrin for valuable comments and discussions. I also would like to deeply thank Thomas Willwacher for communicating the conjectures above and for fruitful discussions. This work was partially supported by the FNR project number: PRIDE $15/10949314/$GSM and the Max Planck Institute for Mathematics in Sciences.

\section{Preliminaries}

\subsection{Notation}

For a natural number $n\in \mathbb N_+$ we will denote by $[n]$ a finite set such that $[n]:=\{1,2,\dots,n\}.$ Let $I$ be a finite set, by $\mathrm {Aut}(I)$ we will denote the group of automorphisms of this set, in the case when $I=[n]$ we will use a notation $\Sigma_n:=\mathrm {Aut}([n])$ for a symmetric group on $n$-letters. We usually work over the field of rational number $\mathbb Q.$ For a $\mathbb Q$-linear representation $V$ of a finite group $G$ we will denote by $V^G$ (resp. $V_G$) the space of $G$-invariants (resp. $G$-coinvariants). Since the characteristic of a field $\mathbb Q$ is zero, the canonical morphism
$V_G\longrightarrow V^G$ is an isomorphism, and hence we will freely switch between invariants and coinvariants. For a finite $S$ we will denote by $\det(S):=\bigwedge^{\dim V\langle S\rangle} V\langle S\rangle$ the determinant on the free $\mathbb Q$-vector space $V\langle S\rangle$ generated by a set $S.$ 
\par\medskip 
For any $g\geq 0,$ and $n\geq 0$ we have a moduli stack $\mathcal M_{g,n}$ which parametrises smooth proper algebraic curves of genus $g$ with $n$ (labelled) marked points. This is a smooth Artin stack of dimension $3g-3+n$ over $\mathbb Z.$ When $2g+n-2>0$ this is a smooth Deligne-Mumford stack. For $2g+n-2>0$ according to \cite{DM} there exists a compactification $\overline {\mathcal M}_{g,n}$ of $\mathcal M_{g,n}$ which is defined by adding all stable nodal curves. $\overline {\mathcal M}_{g,n}$ is a smooth and proper Deligne-Mumford stack over $\mathbb Z.$ The complement to a smooth locus is a divisor with normal crossings. By an abuse of notation we will identify the moduli stack with corresponding coarse moduli spaces. 

\subsection{Homology and Borel-Moore homology} For a locally compact topological space $X$ we will use a notation $D(X)$ for the derived category of sheaves of $\mathbb Q$-vector spaces on $X.$ It is well known that the formalism of six functors exists for the category $D(X).$ By $\omega_X$ we denote the \textit{dualising complex on $X$} i.e. an object in the derived category of sheaves $D(X),$ defined by the rule:
$$
\omega_X:=a_X^!\mathbb Q,\qquad a_X\colon X\longrightarrow *
$$
In particular, when $X$ is a manifold (possible with a corners) of finite dimension we have an isomorphism of sheaves $\omega_X\cong or_X[\dim X],$ here $or_X$ is the sheaf of orientations on $X.$ Note that if $X$ is also an orientable manifold, we have an isomorphism of sheaves $or_X\cong \underline{\mathbb Q}_X.$ In the case when $X$ is a manifold with corners the orientation sheaf $or_X$ is isomorphic to the $!$-extension of the orientation sheaf of the interior $Int(X)$ of $X$ i.e. $or_X\cong i_! or_{Int(X)},$ $i\colon Int(X)\rightarrow X.$ We will use a notation:
$$
\mathbb D\colon D(X)\longrightarrow D(X)^{op},\qquad \mathbb D:=\mathbf R{\EuScript H\mathrm {om}}(\,\,\,,\omega_X)
$$
for the \textit{Verdier duality functor} defined as the internal hom functor with values in the dualising sheaf. Let $\mathcal L$ be a local system on $X$ i.e. a locally constant sheaf, then $\mathbb D(\mathcal L)\cong \mathcal L^{\vee}\otimes \omega_X,$ where $\mathcal L^{\vee}$ is a dual local system. Further, we assume that a space $X$ is equipped with a stratification $\EuScript S=\{S_{\alpha}\}$ $i_{\alpha}\colon S_{\alpha}\rightarrow X.$ For a cohomologically $\EuScript S$-constructible sheaf $\mathcal K\in D(X)$ we define the \textit{chains $C_{\hdotc}(X,\mathcal K)$ of $X$ with coefficients in $\mathcal K$} by the rule:
$$
C_{\hdotc}(X,\mathcal K):=\mathbf R^{-\hdot}\Gamma_c(X,\mathbb D(\mathcal K)).
$$
We also define the \textit{Borel-Moore chains $C^{BM}_{\hdotc}(X,\mathcal K)$ of $X$ with coefficients in $\mathcal K$} by the rule:
$$
C^{BM}_{\hdotc}(X,\mathcal K):=\mathbf R^{-\hdot}\Gamma(X,\mathbb D(\mathcal K)).
$$
When $X$ is a compact manifold, the Borel-Moore chains are quasi-isomorphic to the standard chains with coefficients in a sheaf $\mathcal K:$
$$
C^{BM}_{\hdotc}(X,\mathcal K):=\mathbf R^{-\hdot}\Gamma(X,\mathbb D(\mathcal K))\cong \mathbf R^{-\hdot}\Gamma_c(X,(\mathcal K))^*:=C_{\hdotc}(X,\mathcal K). 
$$
In the case when $\mathcal K$ is a constant sheaf $\underline{\mathbb Q}_X$ we recover the usual definition of the chains and the Borel-Moore chains $C^{BM}_{\hdotc}(X,\mathbb Q):=\mathbf R^{-\hdot}\Gamma(X,\omega_X)).$ When $X$ is a finite-dimensional topological manifold (possible with a boundary) we have \textit{Poincaré-Verdier duality} for the chains and the Borel-Moore chains:
$$
\mathbb D\colon C_{\hdotc}(X,\mathbb Q)\cong C^{\dim X-\hdot}_c(X,or_X),\quad \mathbb D\colon C^{BM}_{\hdotc}(X,\mathbb Q)\cong C^{\dim X-\hdot}(X,or_X)
$$
Where $C^{\hdot}_c(X,\mathcal K)$ are the compactly supported cochains with values in $\mathcal K.$ Note that the Borel-Moore chains are dual to the compactly supported chains i.e. we have a morphism:
$$
C^{BM}_{\hdotc}(X,\mathcal K)^*:=\mathbf R^{-\hdot}\Gamma(X,\mathbb D(\mathcal K))^*\longrightarrow \mathbf R^{\hdot}\Gamma_c(X,\mathcal K):=C_c^{\hdot}(X,\mathcal K).$$
The latter morphism is induced by the natural transformation $\mathbb D^2\longrightarrow \mathrm {Id},$ which is an equivalence for a cohomologically constructible complex of sheaves.
\par\medskip 
Following \cite{KaSh} we define the \textit{homological Cousin complex} i.e. the complex which computes the Borel-Moore homology of $X$ with coefficients in $\mathcal K:$

\begin{Prop}\label{Cous1} The DG-vector space $C_{\hdotc}^{BM}(X,\mathcal K)$ is quasi-isomorphic naturally to (the total DG-vector space arising from) the complex:
$$
\bigoplus_{\cdim S_{\alpha}=0}  C_{\hdotc}^{BM}(S_{\alpha},i_{\alpha}^*\mathcal K)\rightarrow \bigoplus_{\cdim S_{\alpha}=1}C_{\hdotc}^{BM}(S_{\alpha},i_{\alpha}^*\mathcal K)\rightarrow\dots
$$
In particular when all strata $S_{\alpha}$ are contractible the DG-vector space  $C_{\hdotc}^{BM}(X,\mathcal K)$ is quasi-isomorphic to:
$$
\bigoplus_{\cdim S_{\alpha}=0}  K_{\alpha}\otimes \textsf {or}_{{\alpha}}\rightarrow \bigoplus_{\cdim S_{\alpha}=1}  K_{\alpha}\otimes \textsf {or}_{{\alpha}}\rightarrow\dots
$$
Here $K_{\alpha}:=H_{\dim S_{\alpha}}(S_{\alpha},\mathcal K)$ and $\textsf {or}_{\alpha}:=H_{\dim S_{\alpha}}^{BM}(S_{\alpha},\mathbb Q)$ is the one-dimensional space of orientations. Note that the sum over strata of codimension $m$ is placed in degree $-m.$
\end{Prop} 
\begin{proof} Consider the filtration of the space $X:$
$$
X_0\rightarrow X_1\rightarrow\dots \rightarrow X
$$
Here $X_k$ is a union of strata of dimension $k.$ Denote by $Y_k:=X_{k}\setminus X_{k-1}$ with an inclusion $j_k\colon Y_k\rightarrow X.$ With a sheaf $\mathbb D(\mathcal K)$ we can associated the Postnikov tower:
$$
\mathbf R(j_{\dim X*})j_{\dim X}^!\mathbb D(\mathcal K)\longrightarrow \mathbf R(j_{\dim X-1*})j_{\dim X-1}^!\mathbb D(\mathcal K)\longrightarrow \dots \longrightarrow \mathbf R(j_{0*})j_0^!\mathbb D(\mathcal K).
$$
Then applying the derived sections, we get the desired result. 
\end{proof} 

\subsection{Operations on chains and BM-chains} For any morphism $f\colon X\longrightarrow Y$ of locally compact stratified topological spaces and a cohomologically constructible complex of sheaves $\mathcal K$ on $Y$ we define the \textit{pushforward in chains}:
\begin{equation}\label{p1}
f_*\colon C_{\hdotc}(X,f^*\mathcal K)\longrightarrow C_{\hdotc}(Y,\mathcal K)
\end{equation} 
as the derived global sections with compact support of the morphism of sheaves which comes from the following natural transformation of functors:
\begin{equation}\label{adj}
\mathbf R(f_!)f^!\mathbb D\cong \mathbf R(f_!)\mathbb Df^*\longrightarrow \mathbb D
\end{equation}
Further, if a morphism $f\colon X\longrightarrow Y$ is proper, we define the \textit{(proper) pushforward in Borel-Moore chains}:
\begin{equation}\label{bmp1}
f_*\colon C_{\hdotc}^{BM}(X,f^*\mathcal K)\longrightarrow C_{\hdotc}^{BM}(Y,\mathcal K)
\end{equation} 
as the derived global sections of morphism \eqref{adj}. Note that a natural transformation $\mathbf R^{\hdot}\Gamma_c\longrightarrow \mathbf R^{\hdot}\Gamma$ suits in the commutative diagram:
\begin{equation}\label{cc1}
\begin{diagram}[height=2.4em,width=4.9em]
C_{\hdotc}(X,f^*\mathcal K) & \rTo^{f_*} & C_{\hdotc}(Y,\mathcal K)\\
\uTo_{}^{}  &  & \uTo^{}  \\
C_{\hdotc}^{BM}(X,f^*\mathcal K) & \rTo^{f_*} & C_{\hdotc}^{BM}(Y,\mathcal K) \\
\end{diagram}
\end{equation} 
Let $f\colon X\longrightarrow Y$ be an arbitrary morphism of stratified locally compact topological spaces. We define the \textit{Gysin pullback in Borel-Moore chains}:
\begin{equation}\label{bmp2}
f^!\colon C_{\hdotc}^{BM}(Y,\mathcal K)\longrightarrow C_{\hdotc}^{BM}(X,f^!\mathcal K)
\end{equation}
as the derived global sections of the following morphism:
\begin{equation}\label{adj2}
\mathbb D(\mathcal K)\longrightarrow \mathbf R(f_*)f^*(\mathbb D(\mathcal K))\cong \mathbf R(f_*)\mathbb D(f^!\mathcal K)
\end{equation}
In particular, when $j\colon U\rightarrow X$ is an open embedding, we recover a notion of the usual pullback in Bore-Moore chain \cite{CG}. Assume that additionally, the morphism $f$ is proper. We define the \textit{Gysin pullback in chains}:
\begin{equation}\label{bmp2}
f^!\colon C_{\hdotc}(Y,\mathcal K)\longrightarrow C_{\hdotc}(X,f^!\mathcal K)
\end{equation}
as the derived global sections with compact support of morphism \eqref{adj2}. We also have a commutative diagram for Gysin maps:
\begin{equation}\label{cc2}
\begin{diagram}[height=2.4em,width=4.9em]
C_{\hdotc}(Y,\mathcal K) & \rTo^{f^!} & C_{\hdotc}(X,f^!\mathcal K)\\
\uTo_{}^{}  &  & \uTo^{}  \\
C_{\hdotc}^{BM}(Y,\mathcal K) & \rTo^{f^!} & C_{\hdotc}^{BM}(X,f^!\mathcal K) \\
\end{diagram}
\end{equation}

\begin{Lemma}\label{homeqb} Let $f\colon X\longrightarrow Y$ be a morphism of locally compact topological spaces, which is a weak homotopy equivalence and $\mathcal L$ is a local system on $Y.$ The pushforward in chains induces the quasi-isomorphism:
$$
f_*\colon C_{\hdotc}(X,f^*\mathcal L)\overset{\sim}{\longrightarrow} C_{\hdotc}(Y,\mathcal L)
$$ 

\end{Lemma} 
\begin{proof} Denote by $D(\mathcal {L}oc(X))$ a triangulated subcategory of the category $D(X),$ which consists of complexes of $\mathbb Q$-sheaves with locally constant cohomology. One can show that the functor:
\begin{equation}\label{qusl}
f^*\colon D(\mathcal {L}oc(Y))\longrightarrow D(\mathcal {L}oc(X))
\end{equation}
is an equivalence of triangulated categories.\footnote{This follows from the realisation of the category $D(\mathcal {L}oc(Y))$ in terms of "derived monodromy representations", see Appendix A in \cite{Lu}.} Further one has an identification $\mathrm {Ext}^{\hdot}_{D(\mathcal {L}oc(Y))}(\underline{\mathbb Q}_Y,\mathbb D(\mathcal L))^*\cong \mathbf R^{\hdot}\Gamma(X,\mathcal L)^*:=C_{-\hdotc}(X,\mathcal L)$ and the functor $f^*$ induces the morphism:
\begin{align}\label{qusl1}
H^{\hdot}(Y,\mathcal L)&\cong\\
 &\mathrm {Ext}^{\hdot}_{D(\mathcal {L}oc(Y))}(\underline{\mathbb Q}_Y,\mathcal L)\longrightarrow \mathrm {Ext}^{\hdot}_{D(\mathcal {L}oc(X))}(f^*\underline{\mathbb Q}_X,f^*\mathcal L)\cong H^{\hdot}(X,f^*\mathcal L).
\end{align}
The dual to the latter morphism coincides with the pushforward morphism for the chains with coefficients in a local system (this follows from adjunction between the pushforward and the pullback functors). Since functor \eqref{qusl} is equivalence, morphism \eqref{qusl1} is the quasi-isomorphism.

\end{proof}

\begin{remark} For a morphism $f\colon X\longrightarrow Y$ and a cohomologically constructible DG-sheaf $\mathcal K,$ we also have a notion of the Gysin pushforward in the compactly supported chains:
$$
f_!\colon C_c^{\hdot}(X,f^!\mathcal K)\longrightarrow C^{\hdot}_c(Y,\mathcal K)
$$
This morphism is dual to Gysin pullback for the Borel-Moore homology \eqref{bmp2}.
\end{remark} 
We also have the following classical:
\begin{Lemma}\label{gysin} Let $\mathcal K$ be a cohomologically constructible DG-sheaf on a stratified space $X.$ Consider the following diagram of stratified spaces:
$$
j\colon U\hookrightarrow X\hookleftarrow Z\colon i,
$$
where $U$ is an open stratified subspace and $Z$ is the corresponding closed complement. We have the following long exact sequence in Borel-Moore homology:
$$
\dots\rightarrow H^{BM}_{\hdotc}(Z,i^*\mathcal K)\rightarrow H^{BM}_{\hdotc}(X,\mathcal K)\rightarrow H^{BM}_{\hdotc}(U,j^*\mathcal K)\rightarrow\cdots
$$

\end{Lemma} 

\begin{proof} See Subsection $2.6.9$ in \cite{CG}.

\end{proof}

Let $X$ be a topological space equipped with a stratification $\EuScript S.$ We will be mostly interested in the case when $X$ is an orbifold with corners. In that case, we assume that a stratification $\EuScript S$ is defined by a set of corners of $X.$ Following \cite{FOOO1} we have a notion of a smooth fibration of orbifolds with corners. If $f\colon X\longrightarrow Y$ is the smooth fibration of orbifolds with corners we have standard properties for the shriek pullback (Proposition $3.3.2$ \cite{KS}):
\begin{Lemma}\label{KS} Let $f\colon X\longrightarrow Y$ be the smooth fibration of orbifolds with corners. We have the following:
\begin{enumerate}[(i)]
\item There is a canonical quasi-isomorphism of DG-sheaves:
$$
f^*\mathcal K\otimes f^!\underline{\mathbb Q}_X\overset{\sim}{\longrightarrow} f^!\mathcal K.
$$
The DG sheaf $f^!\underline{\mathbb Q}_X$ will be called the relative dualising complex. 

\item The relative dualising complex is isomorphic to a constructible DG-sheaf:
$$
f^!\underline{\mathbb Q}_X\cong \mathbf R(i_*)j_!\mathcal L[\dim X-\dim Y].
$$
Where $\mathcal L$ is a one-dimensional local system on the interior of $Y$ and:
\begin{equation*}\label{}
\begin{diagram}[height=2.4em,width=2.3em]
Int(X) &\rTo^j & Int(Y)\times_X Y   &  & \rTo_{}^{i}   & & X&  \\
  &&\dTo_{}  & &  &  &  \dTo_ {f}  && \\
 && Int(Y) &    & \rTo^{int} &   &  Y   \\
\end{diagram}
\end{equation*}
\end{enumerate}
\end{Lemma}
\begin{proof} Analogous to the case of manifolds. We leave details to the reader.

\end{proof} 
The constructible DG-sheaf $\mathbf R(i_*)j_!\mathcal L$ will be called the \textit{relative orientation sheaf} and denote by $or_f.$ When a local system $\mathcal L$ is trivial, we will say that $f$ is an \textit{oriented smooth fibration of orbifolds with corners}.
\begin{remark}\label{sse} One can describe the constructible DG-sheaf $or_f$ more explicitly. Namely it is a one-dimensional local system on the strata which do not belong to the closure of $Int(Y)\times_X Y\setminus Int(X)$ and it is zero on strata $S\subset \overline{Int(Y)\times_X Y\setminus Int(X)}.$

\end{remark} 

The Gysin pullback along the smooth fibration satisfies the following base change property:

\begin{Prop}\label{bbm} Consider the fibered diagram of stratified spaces:
\begin{equation*}\label{}
\begin{diagram}[height=2.4em,width=2.3em]
 W:=Z\times_X Y   &  & \rTo_{}^{\tilde{f}}   & & Y&  \\
\dTo_{\tilde{g}}  & &  &  &  \dTo_ {g}  && \\
 Z &    & \rTo^{f} &   &  X   \\
\end{diagram}
\end{equation*}
such that the following conditions hold:
\begin{enumerate}[(i)]
\item $g$ is a proper and smooth fibration of orbifolds with corners, 
\item The adjoint to the canonical morphism $\mathbf R(\tilde g_!) \tilde f^*g^!\cong f^*\mathbf R(g_!)g^!\rightarrow f^*$ is an equivalence: 
$$
\tilde{f}^*g^!\overset{\sim}{\longrightarrow} \tilde{g}^!f^*.
$$
\end{enumerate} 
Then for a cohomologically constructible sheaf $\mathcal K$ on $X,$ we have the commutative diagram:
\begin{equation*}\label{}
\begin{diagram}[height=3em,width=2.8em]
C_{\hdotc+d}(W,(g \tilde f)^*\mathcal K\otimes or_{\tilde g})    &  & \rTo_{}^{\tilde{f}_*}   & & C_{\hdotc+d}(Y,g^*\mathcal K\otimes or_g)&  \\
\uTo_{\tilde{g}^!}  & &  &  &  \uTo_ {g^!}  && \\
C_{\hdotc}(Z,f^*\mathcal K) &    & \rTo^{f_*} &   &  C_{\hdotc}(X,\mathcal K).   \\
\end{diagram}
\end{equation*}
Where $\dim Y-\dim X=\dim W-\dim Z:=d$ and $or_{\tilde g}:=\tilde f^*or_{g}.$

\end{Prop}

\begin{proof} Using the standard adjunction morphisms, we have a commutative diagram (in the category of sheaves): 
\begin{equation*}\label{}
\begin{diagram}[height=3em,width=2.6em]
\mathbf R f_!f^!g_!g^*   &  & \rTo_{}^{adj}   & & \mathbf R g_!g^*&  \\
\uTo_{adj}  & &  &  &  \uTo_ {adj}  && \\
\mathbf R f_!f^! &    & \rTo^{adj} &   &  Id   \\
\end{diagram}
\end{equation*}
Now plug in a sheaf $\mathbb D(\mathcal K)$ in the diagram above, we get:
\begin{equation*}\label{}
\begin{diagram}[height=3em,width=2.6em]
\mathbf R f_!f^!g_!g^* \mathbb D(\mathcal K)  &  & \rTo_{}^{adj}   & & \mathbf R g_!g^*\mathbb D(\mathcal K)&  \\
\uTo_{adj}  & &  &  &  \uTo_ {adj}  && \\
\mathbf R f_!f^!\mathbb D(\mathcal K) &    & \rTo^{adj} &   &  \mathbb D(\mathcal K)  \\
\end{diagram}
\end{equation*}
Using the base change Theorem for sheaves (Proposition $3.1.9$ in \cite{KS}), transversality, and the fact that $g$ is a smooth fibration (Lemma \ref{KS}) we get a commutative diagram: 
\begin{equation*}
\begin{diagram}[height=3em,width=2.6em]
\mathbf R f_!\tilde{g}_!\mathbb D(\tilde f^*g^*\mathcal K\otimes or_{[\tilde g}[d])  &  & \rTo_{}^{adj}   & & \mathbf R g_!\mathbb D(g^*\mathcal K\otimes or_g[d])&  \\
\uTo_{adj}  & &  &  &  \uTo_ {adj}  && \\
\mathbf R f_!\mathbb D(f^* \mathcal K) &    & \rTo^{adj} &   &  \mathbb D(\mathcal K)  \\
\end{diagram}
\end{equation*}
Applying the derived sections with compact support, we get the desired diagram.

\end{proof}

\begin{remark}\label{BCC} We will be most interested in the situation when a smooth fibration $g\colon Y\longrightarrow X$ is transversal to an inclusion $f\colon Z\hookrightarrow X$ of stratified orbispaces (cf. \cite{FOOO1} \cite{Gore}). 
In that case, all conditions from Proposition \ref{bbm} are satisfied. 

\end{remark}

Let $\tilde g\colon W\longrightarrow Z$ be a morphism from Proposition \ref{bbm}. We assume that $W$ and $Z$ are compact stratified spaces with contractible strata. We can explicitly describe the Gysin pullback in chains (cf. \cite{Gore}):

\begin{Prop}\label{Gor1} Let $\mathcal L$ be a local system on $Z,$ which is a pullback of some local system on $Y.$ In terms of the Cousin complexes from Proposition \ref{Cous1} the Gysin pullback: 
$$\tilde{g}^!\colon C_{\hdotc}(Z,\mathcal L)\longrightarrow C_{\hdotc+d}(W,\tilde{g}^*\mathcal L\otimes or_{\tilde g})$$
acts by the rule:
$$
\tilde{g}^!\colon H_{0}^{BM}(S_{\alpha},\mathcal L) \longrightarrow H_{0}^{BM}(\widetilde{S}_{\alpha},{\tilde{g}^*\mathcal L\otimes or_{\tilde g}}),
$$
where $\widetilde S_{\alpha}$ is a collection of strata in $\tilde g^{-1}(S_{\alpha})$ of the highest dimension. 
\end{Prop} 

\begin{proof} Consider the adjunction morphism: $adj\colon \mathbb D(\mathcal L) \longrightarrow\mathbf R(\tilde g_*)\tilde g^*\mathbb D(\mathcal L) .$ Using the fact that for the submersion $g$ one has $\tilde g^!\mathcal L\cong \tilde g^*\mathcal L\otimes or_{\tilde g}[\dim X-\dim Y]$ we have:
$$
\mathbb D(\mathcal L) \longrightarrow\mathbf R(\tilde g_*)\mathbb D(\tilde g^!\mathcal L)\cong \mathbf R(\tilde g_*)\mathbb D(\tilde g^*\mathcal L\otimes or_{\tilde g}[\dim X-\dim Y])
$$
Consider the Cousin resolution from Proposition \ref{Cous1} for the sheaf $\mathbf R(\tilde g_*)\mathbb D(\tilde g^*\mathcal L\otimes or_{\tilde g}[\dim X-\dim Y]),$ then from the base change theorem for sheaves \cite{KS}, applied to the diagram:
\begin{equation*}\label{}
\begin{diagram}[height=2em,width=2em]
\tilde g^{-1}(S_{\alpha})   &  & \rTo^{\tilde{j}_{\alpha}}   & & W  \\
\dTo_{\tilde g_{\alpha}}  & &  &  &  \dTo_ {\tilde g}  && \\
S_{\alpha} &    & \rTo^{j_{\alpha}} &   &  Z   \\
\end{diagram}
\end{equation*}
we get the following description of the Gysin pullback morphism in BM chains:
$$
f^!\colon C_{\hdotc}^{BM}(S_{\alpha},i^*_{\alpha}\mathcal L)\longrightarrow C_{\hdotc+\dim X-\dim Y}^{BM}(\tilde g^{-1}(S_{\alpha}),\tilde{i}^*_{\alpha}(\mathcal L\otimes or_{\tilde g}))
$$
Note that by Remark \ref{sse} a sheaf $\tilde{i}^*_{\alpha}(\mathcal L\otimes or_{\tilde g})$ vanishes on strata of dimension less then a dimension of $S_{\alpha}.$ Hence taking the zero homology we get the desired result. 
\end{proof}

\section{Ribbon graph complexes}

\subsection{Ribbon graphs}
By a \textit{ribbon graph} $\Gamma$  we understand a triple $(H(\Gamma), \sigma_1, \sigma_0)$ where $H(\Gamma)$ is a finite set called the \textit{set of half
edges of} $\Gamma$, $\sigma_1\colon H(\Gamma)\longrightarrow  H(\Gamma)$ is a fixed point free involution and a permutation $\sigma_0 \colon H(\Gamma)\longrightarrow  H(\Gamma).$ Orbits of $\sigma_1$ are called the \textit{edges of} $\Gamma$, we will denote the set of edges of $\Gamma$ by $E(\Gamma).$  A set of orbits, $V(\Gamma):= H(\Gamma )/\sigma_0,$ is called the set of \textit{vertices of the ribbon graph} $\Gamma.$ We have a canonical map:
$$
p\colon H(\Gamma)\longrightarrow V(\Gamma):= H(\Gamma )/\sigma_0
$$
For each $v\in V(\Gamma)$ the pre-image of $p$ will be called a set of half-edges attached to a vertex $v:$
$$
p^{-1}(v):=H_v(\Gamma).
$$
The orbits of the permutation $\sigma_{2}:=\sigma_0^{-1}\sigma_1$ are called \textit{boundaries} of the ribbon graph $\Gamma.$ The set of boundaries of $\Gamma$ is denoted by $B(\Gamma)$. By the \textit{genus} of a ribbon graph, we understand the following quantity;
$$
g(\Gamma):=1+\frac{1}{2}(E(\Gamma)-V(\Gamma)-B(\Gamma))
$$
The definition implies that a ribbon graph $\Gamma$ is the same as a standard graph with a fixed cyclic structure on the set of half-edges $H_v(\Gamma)$ at each vertex $v\in V(\Gamma).$ The definitions of a genus and boundaries of a ribbon graph are consistent with the following well-known construction: with a ribbon graph $\Gamma$ it is always possible to associate a compact bordered surface $\Sigma_{\Gamma}$ by replacing vertices with disks and edges with strips \cite{Kon3}. Then a set of boundary components of a ribbon graph corresponds to the set boundaries of the surface $\Sigma_{\Gamma}$ and $g(\Sigma_{\Gamma})=g(\Gamma).$
\par\medskip
Recall that in \cite{MW} for a ribbon graph $\Gamma$ a notion of a \textit{corner} of a ribbon graph is defined. Representing ribbon graphs as graphs with "blown up" vertices a corner is defined as an interval between two half-edges. A set of corners associated with a ribbon graph $\Gamma$ will be denoted by $C(\Gamma).$ For a ribbon graph $\Gamma$ we have the following partition:
$$
C(\Gamma)=\coprod_{b\in B(\Gamma)} C(b)
$$
Elements from $C(b)$ will be called \textit{corners attached to a boundary} $b.$ Following \textit{ibid.} with a ribbon graph $\Gamma$ with $l$-edges and $n$-boundaries we can associate a ribbon graph $e_{c_1,c_2}(\Gamma)$ with $l+1$ edges and $n+1$ boundaries, here $c_1,c_2$ are corners which belong to the \emph{same boundary} $b.$ This operation is defined by attaching a new edge $e$ connecting corners $c_1$ and $c_2.$ Pictorially that means:
$$
\vspace{-7mm} \Ba{c}\resizebox{18mm}{!}
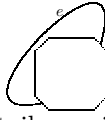
{\xy
(2.5,5.5)*{e},
 (11.5,-1.5)*{}="1",
(-1.5,-14.5)*{}="2",
 (0,0)*{}="a1",
(10,0)*{}="a2",
(13,-3)*{}="a3",
(13,-13)*{}="a4",
(10,-16)*{}="a5",
(0,-16)*{}="a6",
(-3,-13)*{}="a7",
(-3,-3)*{}="a8",
"1";"2" **\crv{(20,27) & (-27,-20)};
\ar @{-} "a1";"a2" <0pt>
\ar @{.} "a2";"a3" <0pt>
\ar @{-} "a3";"a4" <0pt>
\ar @{.} "a4";"a5" <0pt>
\ar @{-} "a5";"a6" <0pt>
\ar @{.} "a6";"a7" <0pt>
\ar @{-} "a7";"a8" <0pt>
\ar @{.} "a8";"a1" <0pt>
\endxy}\vspace{5mm}\Ea
$$
Note that $g(e_{c_1,c_2}(\Gamma))=g(\Gamma),$ for details see \textit{ibid}. 
\subsection{$\textsf {RGra}_d$ and ribbon graph complexes}

Let us briefly recall some results of \cite{MW}. Let $d$ be an integer, denote by $\textsf{RGra}_d$ the \textit{properad of ribbon graphs}. This properad is defined by the collection of vector spaces $\textsf{RGra}_d(n,m)$ which consists of directed and connected ribbon graphs with $[n]$-labelled boundaries and $[m]$-labelled vertices. We assume the choice of the orientation on the set of edges of a ribbon graph (see \textit{ibid}). The composition is defined by attaching a vertex of one ribbon graph to the boundary of another. More precisely with every ribbon graph $\Gamma$ we consider the corresponding bordered Riemann surface $\Sigma_{\Gamma},$ where boundaries of the surface we represent as "out-circle" further we remove disks in $\Sigma_{\Gamma}$ at vertices of $\Gamma$ and represent these boundaries as "in-circle", finally we "glue" two Riemann surfaces in all possible ways by identifying "in-circle" with "out-circles", we refer to \cite{MW} for details (cf. \cite{TZ}). Theorem $4.2.3$ from \cite{MW} (based on \cite{CS1}) claims that there is a natural morphism of properads:
\begin{equation}\label{kp}
s^*\colon \Lambda\textsf{LieB}_{d,d}\longrightarrow \textsf{RGra}_d.
\end{equation}
Where $\Lambda\textsf{LieB}_{d,d}$ is the standard properad controlling Lie $d-1$-bialgebras with a bracket and a cobracket being symmetric operations of degree $d-1.$\footnote{For $d=1$ algebras over this properad are Lie bialgebras in the sense of V. Drinfeld \cite{Drin}.} It will be convenient to represent operations in the properad $\Lambda\textsf{LieB}_{d,d}$ as trees. For the generators of $\Lambda\textsf{LieB}_{d,d}$ it means:
$$
\mathrm {bracket}=\begin{xy}
 <0mm,0.66mm>*{};<0mm,3mm>*{}**@{-},
 <0.39mm,-0.39mm>*{};<2.2mm,-2.2mm>*{}**@{-},
 <-0.35mm,-0.35mm>*{};<-2.2mm,-2.2mm>*{}**@{-},
 <0mm,0mm>*{\circ};<0mm,0mm>*{}**@{},
   <0.39mm,-0.39mm>*{};<2.9mm,-4mm>*{^{}}**@{},
   <-0.35mm,-0.35mm>*{};<-2.8mm,-4mm>*{^{}}**@{},
\end{xy}
\qquad \mathrm{cobracket}=\begin{xy}
 <0mm,-0.55mm>*{};<0mm,-2.5mm>*{}**@{-},
 <0.5mm,0.5mm>*{};<2.2mm,2.2mm>*{}**@{-},
 <-0.48mm,0.48mm>*{};<-2.2mm,2.2mm>*{}**@{-},
 <0mm,0mm>*{\circ};<0mm,0mm>*{}**@{},
 \end{xy}
$$
The morphism $s^*$ is defined on generators of $\Lambda\textsf{LieB}_{d,d}$ by the following rule:
$$
s^*\colon \begin{xy}
 <0mm,0.66mm>*{};<0mm,3mm>*{}**@{-},
 <0.39mm,-0.39mm>*{};<2.2mm,-2.2mm>*{}**@{-},
 <-0.35mm,-0.35mm>*{};<-2.2mm,-2.2mm>*{}**@{-},
 <0mm,0mm>*{\circ};<0mm,0mm>*{}**@{},
   <0.39mm,-0.39mm>*{};<2.9mm,-4mm>*{^{}}**@{},
   <-0.35mm,-0.35mm>*{};<-2.8mm,-4mm>*{^{}}**@{},
\end{xy}
\longmapsto
\xy
 (0,0)*{\bullet}="a",
(5,0)*{\bullet}="b",
\ar @{-} "a";"b" <0pt>
\endxy \ \quad s^*\colon \begin{xy}
 <0mm,-0.55mm>*{};<0mm,-2.5mm>*{}**@{-},
 <0.5mm,0.5mm>*{};<2.2mm,2.2mm>*{}**@{-},
 <-0.48mm,0.48mm>*{};<-2.2mm,2.2mm>*{}**@{-},
 <0mm,0mm>*{\circ};<0mm,0mm>*{}**@{},
 \end{xy}\longmapsto 0.
$$
Hence using the theory of deformation complexes for properads \cite{MV1}\cite{MV2} one may consider the deformation complex of this morphism of properads:\footnote{Note that here we consider ribbon graphs of arbitrary valency.}
\begin{equation}\label{mwg1}
\textsf{RGC}_d^{\hdot}(\delta):=\mathrm{Def}(\Lambda\textsf{LieB}_{d,d}\overset{s^*}{\longrightarrow} \textsf{RGra}_{d})
\end{equation}
According to \cite{MV1} and \cite{MV2} the complex $\textsf{RGC}_d^{\hdot}(\delta)$ is a DG-Lie algebra with a bracket $[-,-]$ which comes from the properadic composition. The differential in complex \eqref{mwg1} is defined by the rule $\delta:=[-,\Gamma],$ here $\Gamma$ is the following element:
$$\Gamma=\xy
 (0,0)*{\bullet}="a",
(5,0)*{\bullet}="b",
\ar @{-} "a";"b" <0pt>
\endxy
$$
We will this object the \textit{Kontsevich-Penner ribbon graph complex}. This complex can be explicitly realised as a complex
with $k$-cochains $\textsf{RGC}^k_d$ being isomorphism classes of ribbon graphs $\Gamma$ with:
$$k=-2gd+|E(\Gamma)|$$
together with an orientation ${or}$ defined by a choice of an element in the vector space:
\begin{enumerate}[(a)]
\item $\det(V(\Gamma))\otimes \det(B(\Gamma))\bigotimes_{e\in E(\Gamma)} \det (H(e))$\footnote{Here by $H(e)$ we have denote the set of two half-edges associated with an edge $e.$} for $d$ odd. 
\par\medskip 
\item $\det(E(\Gamma))$ for $d$ even. 
\end{enumerate} 
We assume that:
$$(\Gamma,{or})=-(\Gamma,{or}^{op}).$$
The differential $\delta\colon \textsf{RGC}_{d}^{k}\longrightarrow \textsf{RGC}_{d}^{k+1}$ is defined by the vertex splitting (which preserves the cyclic structure (see \cite{MW} for details).
\begin{remark} It is evident that for different odd (resp. even) values of $d$ the Kontsevich-Penner ribbon graph complexes are isomorphic up to a shift. Moreover according to \cite{MW} these complexes are isomorphic for \textit{all values} of $d$ (see the proof of Proposition \eqref{expr}).
\end{remark} 
We have the following decomposition:
\begin{equation}\label{sum1}
H^{\hdot}(\textsf{RGC}_d(\delta))=H^{\hdot}(\textsf{RGC}_d(\delta)_{\geq 3})\oplus H^{\hdot}(\textsf{RGC}_d(\delta)_{\leq 2})
\end{equation}
Where by $\textsf{RGC}_d^{\hdot}(\delta)_{\geq 3}$ we have denoted the ribbon graph complex which consist of at least trivalent ribbon graphs and by $\textsf{RGC}_d^{\hdot}(\delta)_{\leq 2}$ we have denoted the ribbon graph complex which consist of ribbon graphs of valency $\leq 2.$ Following the argument from \cite{Will} (Proposition $3.4$) the cohomology of the second direct summand occur only in degrees $1,5,9\dots,2k+1:$
\begin{equation}\label{sum2}
H^{\hdot}(\textsf{RGC}_d(\delta)_{\leq 2})\cong \bigoplus_{i=k\geq 1,\, j\equiv 1 \mod 4} \mathbb Q[j]
\end{equation}
These elements can be represented by the loop-type ribbon graphs $R_{k}$ (ribbon graphs with $k$-edges, $k$-vertices, and two boundaries). Note that the Kontsevich-Penner ribbon graph complex $\textsf{RGC}_d^{\hdot}(\delta)_{\geq 3}$ can be identified with the following product:
\begin{equation}\label{sum3}
\textsf{RGC}_d^{\hdot}(\delta)_{\geq 3}\cong \prod_{g\geq 0,n\geq 1\, 2g+n-2>0}^{\infty} C_c^{\hdot-2gd}(M_{g,n}^{rib}/\Sigma_n,\mathbb Q).
\end{equation}
Where $M_{g,n}^{rib}$ is Kontsevich's moduli space of metric ribbon graphs of genus $g$ with $n$-labelled boundaries (see \cite{Kon1} and \cite{Kon2}).
\par\medskip
Analogous to morphism \eqref{kp} we can consider the morphism $s$ defined by the rule:
$$
s\colon \begin{xy}
 <0mm,0.66mm>*{};<0mm,3mm>*{}**@{-},
 <0.39mm,-0.39mm>*{};<2.2mm,-2.2mm>*{}**@{-},
 <-0.35mm,-0.35mm>*{};<-2.2mm,-2.2mm>*{}**@{-},
 <0mm,0mm>*{\circ};<0mm,0mm>*{}**@{},
   <0.39mm,-0.39mm>*{};<2.9mm,-4mm>*{^{}}**@{},
   <-0.35mm,-0.35mm>*{};<-2.8mm,-4mm>*{^{}}**@{},
\end{xy}
\longmapsto
\xy
 (0,0)*{\bullet}="a",
(5,0)*{\bullet}="b",
\ar @{-} "a";"b" <0pt>
\endxy \ \quad s\colon \begin{xy}
 <0mm,-0.55mm>*{};<0mm,-2.5mm>*{}**@{-},
 <0.5mm,0.5mm>*{};<2.2mm,2.2mm>*{}**@{-},
 <-0.48mm,0.48mm>*{};<-2.2mm,2.2mm>*{}**@{-},
 <0mm,0mm>*{\circ};<0mm,0mm>*{}**@{},
 \end{xy}\longmapsto \ \xy
(0,-2)*{\bullet}="A";
(0,-2)*{\bullet}="B";
"A"; "B" **\crv{(6,6) & (-6,6)};
\endxy.
$$
We can consider the following deformation complex:
\begin{equation}\label{mwg2}
\textsf{RGC}_d^{\hdot}(\delta_s):=\mathrm{Def}(\Lambda\textsf{LieB}_{d,d}\overset{s}{\longrightarrow} \textsf{RGra}_d)
\end{equation}
We will call this object \textit{the Merkulov-Willwacher ribbon graph complex}. Like the Kontsevich-Penner ribbon graph complex the complex
$\textsf{RGC}_d^{\hdot}(\delta_s)$ is a DG-Lie algebra. The differential is defined by the following rule $\delta:=[-,\Gamma],$ here $\Gamma$ is the following element:
$$\Gamma=\xy
 (0,0)*{\bullet}="a",
(5,0)*{\bullet}="b",
\ar @{-} "a";"b" <0pt>
\endxy+
\ \xy
(0,-2)*{\bullet}="A";
(0,-2)*{\bullet}="B";
"A"; "B" **\crv{(6,6) & (-6,6)};
\endxy
$$
The differential $\delta_s$ can be decomposed as $\delta_s=\delta+\Delta_1,$ here: $$\Delta_1:=[-, \ \xy
(0,-2)*{\bullet}="A";
(0,-2)*{\bullet}="B";
"A"; "B" **\crv{(6,6) & (-6,6)};
\endxy]$$
is the so-called \textit{Bridgeland differential}. This morphism can be explicitly described by the rule (up to signs):
\begin{equation}\label{B1}
\Delta_1(\Gamma)=\sum_{b\in B(\Gamma)} \sum_{c_1,c_2\in C(b)}e_{c_1 c_2}(\Gamma).
\end{equation}
The following fact will be important to us:
\begin{Lemma}\label{bd}The morphism $\Delta_1$ satisfies the following properties:
$$
\delta\Delta_1+\Delta_1\delta=0,\qquad \Delta_1^2=0.
$$
\end{Lemma} 
These properties come from the coJacobi identity and a "one cocycle property of a cobracket". Moreover, the Bridgeland differential preserves a genus of a ribbon graph. Hence we have a natural product decomposition:
\begin{equation}\label{ribtot}
\textsf{RGC}_d^{\hdot}(\delta+\Delta_1):=\prod_{g=0}^{\infty} \textsf B_g\textsf{RGC}_d^{\hdot}(\delta+\Delta_1).
\end{equation}
Where we have used the notation $\textsf B_g\textsf{RGC}_d^{\hdot}(\delta+\Delta_1)$ for the direct summand in $\textsf{RGC}_d^{\hdot}(\delta+\Delta_1),$ which consists of ribbon graphs with genus equal to $g.$
\begin{remark}
Definition \eqref{B1} of the morphism $\Delta_1$ was first given by T. Bridgeland \cite{Brid} in the beginning of $00'$s. According to $\textit{ibid.}$ the operation $\Delta_1$ appeared in studying of $N=(2,2)$ SCFT as the dual differential to the standard vertex splitting differential $\delta.$ Namely there is an involution on ribbon graphs given by interchanging vertices with boundaries (see \cite{MW} and \cite{CFL}). Under this involution, the differential $\delta$ corresponds to the Bridgeland differential $\Delta_1.$ Later S. Merkulov and T. Willwacher \cite{MW} rediscovered this operation.\footnote{In our work we also follow the notation from \cite{MW}.}
\end{remark}

\section{Moduli spaces of stable bordered surfaces}

\subsection{Bordered surfaces} There are two classical ways to relate moduli spaces of ribbon graphs to moduli spaces of algebraic curves. One is due to D. Mumford \cite{Mum} \cite{Harer} based on the technique of Jenkins-Strebel differentials \cite{St}, which was later popularised by M. Kontsevich \cite{Kon3} and the second one is due to R. Penner \cite{Pen} based on the hyperbolic geometry. For us, it will be convenient to use the third approach due to K. Costello \cite{Cost}:
\par\medskip
Let $I, K,$ and $J$ be finite sets such that $K\neq \emptyset$ and $\pi\colon I\rightarrow K$ be a morphism of finite sets. Following M. Liu \cite{Liu} For $g\geq 0$ such that $4g+2|K|-4+2|I|+|J|>0$ we denote by $\overline{\mathcal N}_{g,K,\pi,J}^L$ the moduli space of stable bordered surfaces of genus $g$ with $K$-labelled boundaries, $I$-labelled marked points on the boundary such that exactly points $\pi^{-1}(k)$ attached to $k$-boundary and $J$-labelled marked points in the interior of the surface. Here we assume that singularities can be of three types:
\par\medskip
\begin{enumerate}[(a)]
  \item Singularities on the boundary locally isomorphism to:
$$
x^2-y^2=0\quad \mbox{(real node of type one)}.
$$
\par\medskip
 \item Singularities on the boundary i.e. locally isomorphism to:
$$
x^2+y^2=0\quad \mbox{(real node of type two)}.
$$
\par\medskip
  \item Singularities in the interior locally isomorphism to:
  $$xy=0\quad \mbox{(complex singularities)}.$$
\end{enumerate}
$\overline{\mathcal N}_{g,K,\pi,J}^L$ is an orbifold with corners of dimension $6g-6+3|K|+|I|+2|J|.$ The open locus of smooth bordered surfaces will be denoted by $\mathcal N_{g,K,\pi,J}\hookrightarrow \overline{\mathcal N}_{g,K,\pi,J,}^L.$ Note that $Int(\overline{\mathcal N}_{g,K,\pi,J,}^L)$ is given by locus of stable bordered surfaces with at most complex singularities. We will use the notation $\overline{\mathcal N}_{g,|K|,|I|,|J|}^L$ for the following orbifold:
$$
 \overline{\mathcal N}_{g,|K|,|I|,|J|}^L:=\left(\coprod_{\pi\colon I\rightarrow K} \overline{\mathcal N}_{g,K,\pi,J,}^L\right)/ B.
$$
Where the group $B$ is defined by the following rule:
$$
B:=\prod_{k\in K}\mathbb Z_{|\pi^{-1}(k)|}\rtimes \mathrm {Aut}(K)\times \mathrm {Aut}(J).
$$
Note that the quotient orbifold ${\mathcal N}_{g,|K|,|I|,|J|}^L$ is non-orientable, and the corresponding orientation sheaf $or_{{\mathcal N}^L_{g,|K|,|I|,|J|}}$ is defined by the following rule. We consider a one-dimensional representation of a group $B$ defined by the rule (cf. \cite{HVZ}):
\begin{enumerate}
\item It acts as a trivial representation of  $\mathrm {Aut}(J),$
\item It acts as $(-1)^{|\pi^{-1}(k)|-1}$ on generators of cyclic groups,
\item it acts as $(-1)^{(|\pi^{-1}(k)|-1)(|\pi^{-1}(p)-1)}$ on generators of $\mathrm {Aut}(K).$
\end{enumerate} 
Hence the orientation sheaf for the orbifold with corners $\overline{\mathcal N}^L_{g,|K|,|I|,|J|}$ is defined as a $!$-extension of the orientation sheaf $or_{{\mathcal N}^L_{g,|K|,|I|,|J|}}.$ We will denote by $\epsilon_n$ a local system on the orbifold  $\overline{\mathcal N}_{g,n}^L$ which is given by the sign representation of the symmetric group $\Sigma_n.$\footnote{$*$-restrictions of $\epsilon_n$ to different loci in Liu's spaces will be denoted by the same symbol.} We have the canonical quasi-isomorphism:
\begin{equation}\label{orsys1}
\epsilon_n\overset{\sim}{\longrightarrow}\mathbf R(int_*)int^*\epsilon_n
\end{equation} 
This follows from the vanishing of the $!$-restriction of the local system $\epsilon_n$ to a boundary $\partial \overline{\mathcal N}_{g,n}^L$ and the recollement technique for sheaves (\cite{SGA4} Exposé IV). Note that we have an isomorphism of local systems: 
\begin{equation}\label{orsys}
int^*\epsilon_n\cong or_{{\mathcal N}_{g,n}}
\end{equation} 
By $\overline{\mathcal N}_{g,n,k,p}\subset\overline{\mathcal N}_{g,n,k,p}^L$ we will denote the locus of stable bordered surfaces with at most real singularities of the first type, in the sense of K. Costello \cite{Cost}. This is an orbifold with corners of dimension $6g-6+3n+k+2p$  Since, most of the time, we will be interested in surfaces without cusps and marked points on the boundary, we will use the notation $\overline{\mathcal N}_{g,n}$ (resp. ${\mathcal N}_{g,n}$) for $\overline{\mathcal N}_{g,n,0,0}$ (resp. ${\mathcal N}_{g,n,0,0}.$) 
\par\medskip
Analogous to the case of Deligne-Mumford moduli spaces we have morphisms that forget markings. In the case of Liu's moduli spaces, we have morphisms of two types:
\begin{equation}\label{forr}
\pi^{real}_{k}\colon \overline{\mathcal N}_{g,n,k+1,p}^L\longrightarrow \overline{\mathcal N}_{g,n,k,p}^L,
\end{equation}
This morphism forgets the real marked point and stabilise the resulting curve. We also have the complex counterpart of \eqref{forr}:
\begin{equation}\label{forc}
\pi^{comp}_p\colon \overline{\mathcal N}_{g,n,k,p+1}^L\longrightarrow \overline{\mathcal N}_{g,n,k,p}^L
\end{equation}
This morphism forgets the complex marked point and stabilise the resulting curve in the following sense. The possibilities for the nodal surface $\Sigma_{g,n,k,p+1}$ to be non-stable stable after removing a marked point are analogous to the Deligne-Mumford moduli stack \cite{KN}. The first possibility is that $\Sigma_{g,n,k,p+1}$ contains a component $\Sigma_{x,y,z}$ of genus zero with three marked points $x,y$ and $z$ such that $x,y$ are pre-images of nodes under the normalisation map. In that case, we contract this component and reconnect the resulting surface at marked points, being mates of points $x$ and $y.$ The second possibility is that  $\Sigma_{g,n,k,p+1}$ contains a component $\Sigma_{0,x,z}$ of genus zero with two marked points $x$ and $z$ such that $x$ is a pre-image of a node and one boundary component (we allow a boundary component to be a cusp). In that case, we contract this component, and a replace marking, which is a mate of $x$ by a cusp boundary. Note that this map \textit{is not} a universal curve, as it easy to see from its construction. We will use the following replacement: 

Denote by $\widetilde{\mathcal N}_{g,n,k,p+1}^L$ the closed locus in  $\overline{\mathcal N}_{g,n,k,p+1}^L,$ which consists of surfaces without rational component with one boundary of non zero length and two markings such that one of them is a pre-image of the node under the normalisation. We have the induced morphism:
\begin{equation}\label{forcc}
\widetilde{\pi}^{comp}_p\colon \widetilde{\mathcal N}_{g,n,k,p+1}^L\longrightarrow \overline{\mathcal N}_{g,n,k,p}^L
\end{equation}
This map will be proper and defines a universal curve. Let us consider some examples which we will be interested in later on. We will use the following notation for the pullback of morphism \eqref{forcc} along the inclusion of the Costello moduli space:
\begin{equation}\label{forc1}
\pi_{\overline{\mathcal N}_{g,n,0,1}}\colon \overline{\mathcal N}_{g,n,0,1}\longrightarrow \overline{\mathcal N}_{g,n}.
\end{equation}
We will also use the notation:
\begin{equation}\label{forr1}
\pi_{\overline{\mathcal N}_{g,n,2,0}}\colon \overline{\mathcal N}_{g,n,2,0}\longrightarrow \overline{\mathcal N}_{g,n},
\end{equation}
for the pullback of the composition of morphisms \eqref{forr} along the inclusion of the Costello moduli space.

\begin{remark} The moduli spaces of stable bordered surfaces in the sense of Liu are closely related to the moduli spaces of stable real algebraic curves ${}^{\mathbb R}\overline{\mathcal M}_{g,k,2p}$ (see \cite{Cey}, \cite{GM}), \cite{Kap}, \cite{Sep}, \cite{Sil}). The latter classifies stable algebraic curves of genus $g$ with $k$ marked real points and $2p$ marked complex conjugate points. Namely starting with a bordered Riemann surface $\Sigma\in \mathcal N_{g,K,\pi,J}$ we can always associate the complex algebraic curve $\Sigma_{\mathbb C}:=\Sigma \cup_{\partial \Sigma} \overline{\Sigma},$ called the complex double of $\Sigma.$ The complex curve $\Sigma_{\mathbb C}$ has the complex anti-involution $i\colon \Sigma_{\mathbb C}\longrightarrow \Sigma_{\mathbb C}$
and therefore defines the real algebraic curve with some additional structures (labelling of irreducible components (ovals) of $\Sigma_{\mathbb C}(\mathbb R),$ a choice of the "orientation" of $\Sigma_{\mathbb C}(\mathbb R)$ \cite{FOOO}\cite{Cost} (cf. Appendix). This construction extends to a morphism:
\begin{equation}\label{realm}
D\colon \overline{\mathcal N}_{g,n,k,p}^L\longrightarrow {}^{\mathbb R}\overline{\mathcal M}_{2g+n-1,k,2p}
\end{equation}
\end{remark}

\subsection{Shrinking morphisms}

Denote by $\widetilde{\mathcal N}_{g,n}\subset \overline{\mathcal N}_{g,n}^L$ the locus of stable bordered surfaces without real nodes of (a) one and complex nodes (c). The moduli space $\widetilde{\mathcal N}_{g,n}$ is an orbifold with corners of dimension $6g-6+3n$ with an interior being $b\colon \mathcal N_{g,n}\hookrightarrow \widetilde{\mathcal N}_{g,n}.$ By the construction we have an inclusion $a\colon \mathcal M_{g,n}/\Sigma_n\hookrightarrow \widetilde{\mathcal N}_{g,n},$ given by replacing all marked points with cusps. We have the following:

\begin{Th}\label{wp} The inclusion above induces the weak-homotopy equivalence of topological stacks:

$$
a\colon \mathcal M_{g,n}/\Sigma_n\overset{\sim}{\longrightarrow}\widetilde{\mathcal N}_{g,n}.
$$

\end{Th}

\begin{proof} Let $\Sigma_{g,n}$ be a topological surface of genus $g$ and $n$-boundary components. We have the corresponding augmented Teichmüller space ${}^B\overline{\mathcal T}_{g,n}$ (See Appendix Definition \ref{tb4}). We have a canonical morphism $p\colon {}^B\overline{\mathcal T}_{g,n}\longrightarrow \overline {\mathcal N}_{g,n}^L.$ Denote by $\sigma\in C(D\Sigma_{g,n})$ the $n$-simplex which is given by curves lying on the real part of the surface i.e. stable under the diffeomorphism $\tau.$ The Weil-Petersson metric provides us with nearest point projection morphism \eqref{ortp3}:
\begin{equation}\label{ortp4}
\Pi_{{}^B\overline{\mathcal T({\sigma})}}\colon   {}^B\overline{\mathcal T}_{g,n}\longrightarrow {}^B\overline{\mathcal T(\sigma)}.
\end{equation}
Denote by  ${}^B\widetilde{\mathcal T}_{g,n}$ the locus in  ${}^B\overline{\mathcal T}_{g,n}$ which consists of surfaces $(X,f)$ such that $\ell_{\alpha}(X)=0$ iff $\sigma_k$ for any $k\geq 0$, here $\sigma_k$ we denote the $k$-simplex with curves lying in the real part of the surface $D\Sigma_{g,n},$ in particular for $k=0$ we assume that this set is empty. 

\begin{Lemma} The pullback of the nearest point projection \eqref{ortp4} along the inclusion ${}^B{\mathcal T(\sigma)}\hookrightarrow {}^B\overline{\mathcal T(\sigma)}$ is equivalent to ${}^B\widetilde{\mathcal T}_{g,n}:$
\begin{equation*}\label{}
\begin{diagram}[height=2.3em,width=2.1em]
{}^B\widetilde{\mathcal T}_{g,n} &  & \rTo^{}   & &  {}^B\overline{\mathcal T}_{g,n}&  \\
\dTo_{}  & &  &  &  \dTo_ {\Pi_{{}^B\overline{\mathcal T({\sigma})}}}  && \\
{}^B{\mathcal T(\sigma)} &    & \rTo^{} &   &  {}^B\overline{\mathcal T(\sigma)}   \\
\end{diagram}
\end{equation*}
\end{Lemma} 
\begin{proof} Building on Theorem $4.18$ from \cite{wolp1} and following the scheme of Lemma $4.4$ from \cite{Fuj}  (see also \cite{Yam2}) one can show that for $x\in \mathcal T(\sigma_k)^{\tau}$ and for any natural numbers $l$ and $p$ such that $k<l<p,$ the Weil-Petersson metric satisfies the property:
$$
d_{WP}(x,\mathcal T(\sigma_l)^{\tau})<d_{WP}(x, \mathcal T(\sigma_p)^{\tau}),\qquad 
$$
Then the desired result follows.

\end{proof} 
Taking the composition of the inclusion ${}^B\mathcal T_{g,n}\hookrightarrow {}^B\widetilde{\mathcal T}_{g,n}$ with the pullback morphism of \eqref{ortp4} we get the morphism:
\begin{equation}\label{ortp5}
\Pi_{{}^B{\mathcal T({\sigma})}}\colon   {}^B{\mathcal T}_{g,n}\longrightarrow {}^B{\mathcal T(\sigma)},
\end{equation}
which defines the homotopy equivalence. Denote by $MCG_{g,n,\sigma}^B$ the subgroup of the mapping class group of the bordered surface which consists of elements $g$ such that $g(\gamma)$ is homotopic to $\gamma$ for any $\gamma\in \sigma$ and $g$ fixes each component of $D\Sigma_{g,n}\setminus \sigma.$ We have the following:
\begin{Lemma} The canonical inclusion of groups $MCG_{g,n,\sigma}^B\longrightarrow MCG_{g,n}^B$ is an isomorphism. 
\end{Lemma}

\begin{proof} 
The proof directly follows from the presentation of the mapping class group in terms of Dehn twists. 

\end{proof}
From this Lemma, we have that morphism \eqref{ortp5} is an $MCG_{g,n}^B$-equivariant morphism between contractible spaces, such that the mapping class group acts properly discontinuous. Hence it induces the homotopy equivalence of the corresponding topological stacks\footnote{Since the group $MCG_{g,n}^B$ acts properly discontinues the corresponding quotient stacks are topological Deligne-Mumford stack \cite{Nohi} Corollary $14.6$}:
\begin{equation}\label{ortp6}
\Pi_{\mathcal M_{g,n}/\Sigma_n}\colon   \mathcal N_{g,n}\cong [{}^B{\mathcal T}_{g,n}/ MCG_{g,n}^B]\overset{\sim}{\longrightarrow} [{}^B{\mathcal T(\sigma)}/ MCG_{g,n}^B]\cong\mathcal M_{g,n}/\Sigma_n
\end{equation}
Recall that a homotopy type\footnote{The homotopy type of a topological stack $\mathcal X$ equipped with an atlas $f\colon U\longrightarrow \mathcal X$ is defined as a homotopy type of the total space of the simplicial space:
$\mathcal X^{\Delta}:=\{U\times_{\mathcal X} \dots\times_{\mathcal X} U\}.$} of a quotient stack $[X/G]$ is given by the Borel construction $EG\times_G X$ \cite{Nohi}. Applying this we get the morphism:
\begin{equation}\label{hompb}
EMCG_{g,n}^B\times_{MCG_{g,n}^B} {}^B{\mathcal T}_{g,n}\longrightarrow EMCG_{g,n}^B\times_{MCG_{g,n}^B} {}^B{\mathcal T(\sigma)}
\end{equation}
Morphism \eqref{hompb} is a homotopy equivalence since \eqref{ortp5} is a homotopy equivalence. 
\par\medskip 
We have the following diagram of topological stacks:
$$
[{}^B{\mathcal T(\sigma)}/ MCG_{g,n}^B]\longrightarrow [{}^B\widetilde{\mathcal T}_{g,n}/MCG_{g,n}^B]\longleftarrow  [{}^B{\mathcal T}_{g,n}/ MCG_{g,n}^B]
$$
Note that the left inverse to the morphism $a:=i_{\sigma}$ is $\widetilde{\Pi}_{{}^B{\mathcal T({\sigma})}}.$ Since the composition $b\circ \widetilde{\Pi}_{{}^B{\mathcal T({\sigma})}}$ is a homotopy equivalence, by two-out-of-three property of weak homotopy equivalences we get that $a$ is a homotopy equivalence as well.

\end{proof}

Applying Lemma \ref{homeqb} to the Theorem above, we give: 
\begin{Def} For every $g\geq 0$ and $n\geq 1$ such that $2g+n-2>0$ we define the \textit{boundary shrinking morphism}:
\begin{equation}\label{shr1}
\rho_{\mathcal M_{g,n}/\Sigma_n}\colon C_c^{\hdot}(\mathcal N_{g,n},\mathbb Q)\overset{\sim}{\longrightarrow} C_c^{\hdot-n}(\mathcal M_{g,n}/\Sigma_n,\epsilon_n)\end{equation}
as the zigzag quasi-isomorphism:
$$
 C_c^{\hdot}(\mathcal M_{g,n}/\Sigma_n,\epsilon_n)\overset{a_!}{\overset{\sim}{\longrightarrow}} C_c^{\hdot+n}(\widetilde{\mathcal N}_{g,n},\mathbb D(\epsilon_n))\overset{b_!}{\overset{\sim}{\longleftarrow}} C_c^{\hdot+n}({\mathcal N}_{g,n},\mathbb Q)
$$
\end{Def}

\begin{remark}\label{topbsh} Let us explain the coefficients in the definition of the boundary shrinking morphism. We consider the sign local system $\epsilon_n$ on Liu's moduli space of stable curves $\widetilde{\mathcal N}_{g,n}.$ By the construction above, we have a quasi-isomorphism:
$$
C_{\hdotc}(\mathcal M_{g,n}/\Sigma_n,\epsilon_n)\overset{a_*}{\overset{\sim}{\longrightarrow}} C_{\hdotc}(\widetilde{\mathcal N}_{g,n},\epsilon_n)\overset{b_*}{\overset{\sim}{\longleftarrow}} C_{\hdotc}({\mathcal N}_{g,n},\epsilon_n)$$ 
Applying Poincaré-Verdier duality on both sides and noticing that $\mathcal M_{g,n}/\Sigma_n$ is orientable while the orientation sheaf for the interior $\mathcal N_{g,n}$ coincides with a sign local system \eqref{orsys} we get the desired result. Moreover, the boundary shrinking morphism can be identified with the Gysin pushforward along the "topological shrinking morphism" $\Pi_{\mathcal M_{g,n}/\Sigma_n}.$

\end{remark}

\begin{remark} Note that morphism \eqref{shr1} can be generalised to the case of moduli spaces $\mathcal N_{g,I,\emptyset,J}$ (moduli spaces of smooth bordered surfaces of genus $g$ with $I$-boundary components and $J$ marked points in the interior). The constructions from the proof of Theorem \ref{wp} can be generalised to this case since the Weil-Petersson metric exists for the moduli spaces with marked points (we replace non-marked Teichmüller space with the marked one in Definition \ref{tb1}). More precisely, we define the shrinking morphism in the case of the moduli spaces of non-labelled boundaries and labelled marked points and take a pullback along the corresponding quotient covering. Hence we get a morphism:
\begin{equation}\label{shr}
\rho_{\mathcal M_{g,I\sqcup J}}\colon C_c^{\hdot}(\mathcal N_{g,I,0,J},\mathbb Q)\overset{\sim}{\longrightarrow} C_c^{\hdot-|I|}(\mathcal M_{g,I\sqcup J},\mathbb Q)
\end{equation}

\end{remark}

\begin{remark} There is an alternative (implicit) way to show a homotopy equivalence between $\mathcal M_{g,n}$ and $\mathcal N_{g,[n]}.$ Consider the moduli space of smooth algebraic curves with a choice of a parametrisation around each marked point $\mathcal M_{g,n^{\infty}}$ \cite{Kon5} \cite{BS}. This moduli space is a smooth pro-algebraic variety, equipped with a forgetful morphism $\mathcal M_{g,n^{\infty}}\longrightarrow \mathcal M_{g,n}$ which is $\mathrm {Aut}(\mathbf D)^{\times n}$-torsor, where $\mathbf D:=\mathrm {Spf}(\mathbb C[[z]])$ is a formal disk. Denote by $\mathcal P_{g,[n]}$ the moduli space of smooth bordered surfaces of genus $g$ with $[n]$-labelled analytically parametrised boundaries. This space is also equipped with a forgetful morphism $\mathcal P_{g,[n]}\longrightarrow \mathcal N_{g,[n]}$ which is $\mathrm {Diff}(S^1)^{\times n}$-torsor. Analytic gluing of a marked disk defines a homotopy equivalence $sew\colon \mathcal P_{g,[n]}\longrightarrow \mathcal M_{g,n^{\infty}},$ which is compatible with group actions. Applying the Serre long exact sequence for "stacky" homotopy groups (Theorem $5.2$ \cite{Noohi}) we define a morphism between the "stacky" fundamental groups $\pi_1^{orb}(\mathcal N_{g,[n]})\longrightarrow \pi_1^{orb}(\mathcal M_{g,[n]})$ as being induced by $sew$ and further get the resulting morphism between classifying stacks.

\end{remark}

\subsection{Costello's homotopy equivalence}

Analogous to the case of the Deligne-Mumford moduli stacks the orbifold $\overline{\mathcal N}_{g,n}$ admits a stratification in terms of dual graphs. Fix natural numbers $g\geq 0$ and $n\geq 1$ such that $2g+n-2>0.$ Denote by $J_{2g+n-1}$ a category of stable weighted graphs of genus $2g+n-1$ \cite{KG} \cite{CGP1}. For an element $G\in J_{2g+n-1}$ we will we denote by ${}^{\mathbb R}\overline{\mathcal M}_{2g+n-1}(G)\subset {}^{\mathbb R}\overline{\mathcal M}_{2g+n-1}$  the locus of stable real curves with a dual graph being exactly $G,$ we assume that edges of a graph $G$ correspond to real nodes of type (a). By $\mathcal N(G)\subset \overline{\mathcal N}_{g,n}$
we denote the preimage of a real stratum under morphism \ref{realm}:
$$\mathcal N(G):=D^{-1}\left({}^{\mathbb R}\overline{\mathcal M}_{2g+n-1}(G)\right).$$
Note that for some stable weighted graphs $G$ the corresponding stratum in $\overline{\mathcal N}_{g,n}$ is empty. 
\par\medskip 
We denote by $D_{g,n,k,p}\subset \overline{\mathcal N}_{g,n,k,p}$ the locus of the stable bordered Riemann surfaces with irreducible components given by disks and the number of complex marked points at each component of the normalisation is required to be less or equal to one. Just like in the case of the space $\overline{\mathcal N}_{g,n,k,p}$ we will use the special notation $D_{g,n}$ for $D_{g,n,0,0}.$ Note that $D_{g,n,k,p}$ is the compact topological stack of the following dimension (See the proof of Proposition \ref{expr}):
$$
\dim_{\mathbb R} D_{g,n,k,p}=4g-5+2n+k+2p.
$$
It is naturally stratified by the stratification induced from $\overline{\mathcal N}_{g,n}.$
\begin{Prop}\label{expr} We have the following (well-known) quasi-isomorphism:
$$
 \textsf{RGC}^{\hdot}_d(\delta)_{\geq 3}\overset{\sim}{\longrightarrow}\prod_{g\geq 0,n\geq 1, 2g+n-2>0}^{\infty} C_{\dim \mathcal N_{g,n}-2dg-\hdotc}(D_{g,n},\epsilon_n).
$$
\end{Prop}

\begin{proof} Denote by $J_{2g+n-1}^{w=0}$ the subcategory of the category $J_{2g+n-1}$ which consists of stable weighted graphs with zero weights assigned to vertices i.e. vertices are at least trivalent. For an element $G\in J_{2g+n-1}^{w=0}$ the stratum $\mathcal N(G)$ admits the following description: 
\begin{equation}\label{strat1}
\mathcal N(G)\cong \left(\prod_{v\in V(G)} \mathcal N_{0,1,H_{v},0}\right)/\mathrm {Aut}(G). 
\end{equation}
Where by $H_v$ we denote the set of half-edges attached to a vertex $v.$ Note that $\dim \mathcal N(G)=H(G)-3V(G).$ Applying the Cousin complex (Proposition \ref{Cous1}) we get that the complex $C_{\hdotc}(D_{g,n},\epsilon_n)$ is quasi-isomorphic to:
$$
\bigoplus_{\cdim \mathcal N(G)=0}  C_{\hdotc}^{BM}(\mathcal N(G),\epsilon_n) \rightarrow \bigoplus_{\cdim \mathcal N(G)=1}  C_{\hdotc}^{BM}(\mathcal N(G),\epsilon_n)\rightarrow\dots
$$
Note that a moduli space  $\mathcal N_{0,1,H_v,0}$ is naturally equivalent to the configuration space $Conf^{|H_v|}(S^1)/ PSL_2(\mathbb R).$ For $|H_v|\geq 3$ it is easy to see that the configuration space $Conf^{|H_v|}(S^1)/ PSL_2(\mathbb R)$ is homeomorphic to the union of copies of $\mathbb R^{|H_v|-3}$ labelled by number of cyclic orders on the set $H_v$ (cf. \cite{Sta}). Hence each stratum \eqref{strat1} is an orbicell and $\dim D_{g,n}=4g-5+2n$ (since the minimal stratum has dimension $2g+n-1$, corresponding to a stable graph with unique vertex of genus zero and $2g+n-1$ loop attached). The connected components of stratum \eqref{strat1} correspond to the "decorations" of $G$ by a structure of a ribbon graph of genus $g$ with $n$ boundary components.  Note that sections of a local system $\epsilon_n$ over a connected component corresponding to a ribbon graph $\Gamma$ are given by $\det(B(\Gamma)),$ hence we get that entries in the complex above are given by pairs $(\Gamma,or)$ where $\Gamma$ is a ribbon graph (placed in degree $6g-6+3n-|E(\Gamma)|$ and $or$ is an element in the following space:
\begin{align}\label{or1}
\otimes_{v\in V(\Gamma)}&H_0^{BM}(\mathbb R^{|H_v|-3},\mathbb Q)\otimes \det(B(\Gamma))\\
&\cong \det(V(\Gamma))\otimes \det (H(\Gamma))\otimes \det(B(\Gamma)).
\end{align} 
Further following \cite{CFL} (see also \cite{MW}) we explicitly compute:
\par\medskip
Recall that by $H_e(\Gamma)$ we have denoted a set of two half edges contained in the edge $e\in E(\Gamma).$ We can identify canonically $\det(H(\Gamma))\cong \otimes_{e\in E(\Gamma)} \det H_e(\Gamma).$ Consider the bordered Riemann surface $\Sigma_{\Gamma}$ associated with a graph $\Gamma.$ By gluing disks to the boundary components of $\Sigma_{\Gamma}$ we obtain the closed Riemann surface denoted by $\widehat{\Sigma}_{\Gamma}$ Using the cellular decomposition which comes from a ribbon graph $\Gamma$ we have the following complex which computes singular homology of $\widehat{\Sigma}_{\Gamma}$ with $\mathbb Q$-coefficients:
\begin{equation}\label{cellcomp}
\begin{diagram}[height=3em,width=3em]
C_2(\widehat{\Sigma}_{\Gamma},\mathbb Q) &   \rTo^{d_2}   &  C_1(\widehat{\Sigma}_{\Gamma},\mathbb Q) &   \rTo^{d_1}   & C_0(\widehat{\Sigma}_{\Gamma},\mathbb Q)
\end{diagram}
\end{equation}
Here the space of $2$-chains can be identified with space of boundaries $C_2(\widehat{\Sigma}_{\Gamma},\mathbb Q)=\mathbb Q[B(\Gamma)]$, space of $1$-chains with set of half-edges $C_1(\widehat{\Sigma}_{\Gamma},
\mathbb Q)=\summm_{e\in E(\Gamma)} \mathbb Q[H_e(\Gamma)]$ and $0$-chains with vertices $C_0(\widehat{\Sigma}_{\Gamma},\mathbb Q)=\mathbb Q[V(\Gamma)].$ We have the following exact sequences:
$$
0 \rightarrow H_2(\widehat{\Sigma}_{\Gamma},\mathbb Q)\rightarrow C_2(\widehat{\Sigma}_{\Gamma},\mathbb Q)\rightarrow \mathrm {Im}(d_2)\rightarrow 0,
$$

$$
0 \rightarrow \mathrm{Im}(d_2)\rightarrow \mathrm {Ker}(d_1)\rightarrow H_1(\widehat{\Sigma}_{\Gamma},\mathbb Q)\rightarrow 0,
$$

$$
0 \rightarrow \mathrm {Ker}(d_1)\rightarrow C_1(\widehat{\Sigma}_{\Gamma},\mathbb Q)\rightarrow C_0(\widehat{\Sigma}_{\Gamma},\mathbb Q)\rightarrow H_0(\widehat{\Sigma}_{\Gamma},\mathbb Q)\rightarrow 0,
$$
Further, using Lemma $1$ from \cite{CV} and identifications above we have the following isomorphisms:
$$
\det(B(\Gamma))\cong \det(H_2(\widehat{\Sigma}_{\Gamma},\mathbb Q))\otimes \det(\mathrm {Im}(d_2))
$$

$$
\det (\mathrm {Ker}(d_1))\cong \det(\mathrm{Im}(d_2))\otimes \det(H_1(\widehat{\Sigma}_{\Gamma},\mathbb Q))
$$

$$
\det (\mathrm {Ker}(d_1))\otimes \det (V(\Gamma))\cong \det (E(\Gamma))\otimes \bigotimes_e \det(H_e(\Gamma))\otimes \det(H_0(\widehat{\Sigma}_{\Gamma},\mathbb Q))
$$
Note that since spaces $H_0(\widehat{\Sigma}_{\Gamma},\mathbb Q)$ and $H_2(\widehat{\Sigma}_{\Gamma},\mathbb Q)$ are one dimensional and have a canonical basis we can omit them from expressions above. Further vector space $H_1(\widehat{\Sigma}_{\Gamma},\mathbb Q)$ can be equipped with a basis of cycles $\{\gamma_i,\delta_i\}_{i\in K}$ where $|K|=g,$ hence can be also omitted. Hence we have:
\begin{equation}\label{or2}
 \det(E(\Gamma))\cong\det(V(\Gamma))\otimes \det (B(\Gamma))\otimes \bigotimes_e \det(H_e(\Gamma))
\end{equation}
Comparing the right-hand side of \eqref{or2} with the \eqref{or1} we obtain the desired identification of orientation spaces.
\par\medskip
Finally, it is easy to see that the differential $C_{\hdotc}(D_{g,n},\epsilon_n)$ coincides with the differential $\delta$ in the Kontsevich-Penner ribbon graph complex. Hence we get:
$$
C_{6g-6+3n-\hdotc}(D_{g,n},\epsilon_n)\cong C^{\hdot}_c(M_{g,n}^{rib}/\Sigma_n,\mathbb Q)
$$
\end{proof}
\begin{remark} Following \cite{Cost} Proposition \ref{expr} can be extended to the case of moduli spaces of nodal surfaces with markings in the interior and markings on the boundary. Namely, in these cases, we replace the notion of a ribbon graph with a notion of a ribbon graph with black and white vertices (resp. ribbon graphs with hairs), where white vertices (resp. hairs) correspond to components which have a complex marking (resp. marking on a boundary).

\end{remark}

The main result (Theorem $2.2.2$) of \cite{Cost} gives us a \textit{weak equivalence of orbispaces}:\footnote{The fact that $v_{g,n,k,p}$ is an equivalence is obvious, since $\mathcal N_{g,n,k,p}$ is an interior of $\overline{\mathcal N}_{g,n,k,p}.$}
\begin{equation}\label{rib2}
v_{g,n,k,p}\colon \mathcal N_{g,n,k,p}\longrightarrow \overline{\mathcal N}_{g,n,k,p}\longleftarrow D_{g,n,k,p}\colon u_{g,n,k,p}.
\end{equation}
Further, we will use a notation $v$ (resp. $u$) for a morphism $v_{g,n,k,p}$ (resp. $u_{g,n,k,p}$) if it will not lead to confusion. We have the following (well known):
\begin{Prop}\label{costl} For $d\in \mathbb Z$ the Costello weak equivalence together with quasi-isomorphism \eqref{shr1} define the zig-zag quasi-isomorphism:
\begin{equation}
  \textsf{cost}\colon \prod_{g\geq0,n\geq 1, 2g+n-2>0}^{\infty}\big(C_c^{\hdot+2dg-n}(\mathcal M_{g,n},\mathbb Q)\otimes_{{\Sigma_n}} \mathrm {sgn}_n\big)^{\Sigma_n}{\overset{\sim}{\longrightarrow}}\textsf{RGC}_d^{\hdot}(\delta)_{\geq 3}.
\end{equation}

\end{Prop}
\begin{proof} Applying Lemma \ref{homeqb} we define the morphism $\textsf{cost}$ as the zigzag quasi-isomorphism:

\begin{align*}
 (C_c^{\hdot-n}(\mathcal M_{g,n},\mathbb Q)\otimes \mathrm{sgn}_n)^{\Sigma_{n}}&\overset{\mathbb D\rho_{\mathcal M_{g,n}/\Sigma_n}}{\overset{\sim}{\longrightarrow}} C_{\dim \mathcal M_{g,n}-\hdotc}(\mathcal N_{g,n},\epsilon_n)\overset{v_*}{\overset{\sim}{\longrightarrow}} C_{\dim \mathcal M_{g,n}-\hdotc}(\overline{\mathcal N}_{g,n},\epsilon_n)\\
 &\overset{u_*}{\overset{\sim}{\longleftarrow}} C_{\dim \mathcal M_{g,n}-\hdotc}(D_{g,n},\epsilon_n){\overset{\sim}{\longleftarrow}} C^{\hdot}_c(M_{g,n}^{rib}/\Sigma_n,\mathbb Q).
\end{align*}

\end{proof}

\section{Proofs of T. Willwacher's and A. Caldararu's conjectures} 

\subsection{T. Bridgeland's differential $\Delta_1^{bor}$}

For every $g\geq 0$ and $n\geq1,$ such that $2g+n-2>0$ by $T_{g,n}$ we denote the locus in $\overline {\mathcal N}_{g,n}$ which consists of stable bordered surfaces such that all irreducible components are disks except the unique component which is an annulus (we also assume that all nodes are connected to the one boundary of the annulus). By $K_{g,n}^L$ we denote the closure of the union of $D_{g,n}$ with $T_{g,n}$ in Liu's moduli space:
$$
K_{g,n}^L:=\overline{T_{g,n}\cup D_{g,n}}\subset \overline{\mathcal N}_{g,n}^L
$$
Note that by the construction the locus $K_{g,n}^L$ is a compact topological stack. We have the following diagram:
$$
\xi_{D_{g,n,0,1}}\colon D_{g,n,0,1}\hookrightarrow K_{g,n+1}^L\hookleftarrow D_{g,n+1}\colon n
$$
Where the morphism $\xi_{D_{g,n,0,1}}$ is defined by replacing a marked point with a cusp and $n$ is a closed inclusion. We have the following:
\begin{Lemma}\label{B0} The push-forward along the morphism $n$ induces the quasi-isomorphism:
$$
n_*\colon C_{\hdotc}(D_{g,n+1},\epsilon_{n+1})\overset{\sim}{\longrightarrow} C_{\hdotc}(K_{g,n+1}^L,\epsilon_{n+1})
$$

\end{Lemma}
\begin{proof} Consider the Gysin triangle (Lemma \ref{gysin}):
$$
C_{\hdotc}(D_{g,n+1},\epsilon_{n+1})\longrightarrow C_{\hdotc}(K_{g,n+1}^L,\epsilon_{n+1})\longrightarrow C_{\hdotc}^{BM}(K_{g,n+1}^L\setminus D_{g,n+1},\epsilon_{n+1})\longrightarrow
$$
in the Borel-Moore chains associated with the diagram:
$$
n\colon D_{g,n+1}\rightarrow K_{g,n+1}^L\hookleftarrow K_{g,n+1}^L\setminus D_{g,n+1}
$$
In order to prove that the morphism $n_*$ is the quasi-isomorphism it is enough to show that the complex: $$C_{\hdotc}^{BM}(K_{g,n+1}^L\setminus D_{g,n+1},\epsilon_{n+1})$$ is acyclic. We have a sequence of inclusions:
$$
T_{g,n}\longrightarrow K_{g,n+1}^L\setminus D_{g,n+1}\longleftarrow T_{g,n+1}^{cusp}
$$
Here the inclusion $T_{g,n}\hookrightarrow K_{g,n+1}^L$ is open and $T_{g,n+1}^{cusp}$ is the corresponding closed complement. Note that a pullback of a local system $\epsilon_{n+1}$ on $T_{g,n+1}^{cusp}$ under this morphism can be naturally identified with $\epsilon_n$ via the morphism: 
\begin{equation}\label{loc1}
\xi_{D_{g,n,0,1}}^*\epsilon_{n+1}\longrightarrow \epsilon_n,
\end{equation}
defined by the rule: 
$$
\delta'''\colon b'\wedge\dots\wedge b''\wedge \dots\wedge b'''\longrightarrow (-1)^{j}b'\wedge\dots\wedge b'''.
$$
Here $b''$ is a unique boundary without attached nodes placed in the $j$-slot in the sequence above. Morphism \eqref{loc1} is the non-zero morphism between one-dimensional local systems and hence it is the isomorphism. Hence the complex $C_{\hdotc}^{BM}(K_{g,n+1}^L\setminus D_{g,n+1},\epsilon_{n+1})$ is quasi-isomorphic to the mapping cone of the following morphism:
$$
\delta'''\colon C_{\hdotc}^{BM}(T_{g,n+1},\epsilon_{n+1})\longrightarrow C_{\hdotc-1}(D_{g,n,0,1},\epsilon_{n})
$$
Analogous to Proposition \ref{expr} we can realise the complex $C_{\hdotc}(D_{g,n,0,1},\epsilon_{n})$ as complex of pairs $(\Gamma,\alpha).$ Here $\Gamma$ is a ribbon graph with a unique "white vertex" \cite{Cost}. The orientation $\alpha$ is given by choice of an element in:
$$
\det(V_{\bullet}(\Gamma))\otimes \det(H(\Gamma))\otimes \det (B(\Gamma))
$$
Here $V_{\bullet}(\Gamma)$ is a set of black vertices. The differential is defined by splitting a vertex. Analogous to $D_{g,n,0,1}$ we can stratify $T_{g,n+1}$ by associating with a surfaces $\Sigma_{g,n+1}\in  T_{g,n+1}$ the corresponding dual graph $\Gamma$ (as an element in the moduli space of real curves). The unique vertex $v$ of genus $1$\footnote{The unique genus one vertex comes from considering an annulus as a real curve.} will be called the special vertex and denoted by $v_{\emptyset}.$ Hence the corresponding stratum is defined by the rule:
$$
S(\Gamma):=\left(\prod_{v\in V_{w(0)}(\Gamma)} X_v\times \mathcal N_{0,2,\pi}\right)/\mathrm {Aut}(\Gamma)
$$
Here $\pi\colon H_{v_{\emptyset}}\rightarrow [2]$ is a non-surjective morphism and $V_{w(0)}(\Gamma)$ is a set of vertices with zero weights attached. Hence we can compute the dimension of a stratum $S(\Gamma):$
$$
|\sqcup_{v\in V_{w(0)}} H_v|-3|V_{w(0)}|+\dim_{\mathbb R} \mathcal N_{0,2,\pi}=|H(\Gamma)|-3|V_{w(0)}|
$$
Note that according to \cite{DHV} the moduli space $\mathcal N_{0,2,\pi}$ is an interior of convex polytope and it is homeomorphic to the union of $\mathbb R^{H_{v_{\emptyset}}}$ over all cyclic orders on $H_{v_{\emptyset}}.$ Hence we compute:
$$
\otimes_{v\in } H_0(X_v)\otimes H_0(\mathbb R^{H_{v_{\emptyset}}})\cong \det(V_{w(0)}(\Gamma)\otimes \det (H(\Gamma))
$$
Thus we may identify chains in $C_{\hdotc}^{BM}(T_{g,n+1},\epsilon_{n+1})$ with pairs $(\Gamma,\widetilde{or})$ consisting of a ribbon graph $\Gamma$ with a unique white vertex and $\widetilde {or}\in \det(V_{\bullet}(\Gamma))\otimes \det(H(\Gamma))\otimes \det(B(\Gamma)\sqcup V_{\circ}(\Gamma)).$ Let $Rib^{k}_{g,n,1}$ be a set of ribbon graphs of genus $g$ with $n$ boundary components and with a unique white vertex. We can explicitly realise $C_{\hdotc}^{BM}(K_{g,n+1}^L\setminus D_{g,n+1},\epsilon_{n+1})$ as the total DG-vector space of the following complex:
\par\medskip
\begin{equation*}
\begin{diagram}[height=3em,width=3.3em]
\vdots & &  & &  \vdots  &  \\
\dTo^{} & & & & \dTo^{} & &     &  \\
\bigoplus_{\Gamma\in Rib^{k}_{g,n,1}}(\Gamma,\widetilde{\alpha}) && \rTo^{\delta'''}    &  & \bigoplus_{\Gamma\in Rib^{k}_{g,n,1}}(\Gamma,\alpha)  &   \\
\dTo^{} & & & & \dTo^{} & &     &  \\
\bigoplus_{\Gamma\in Rib^{k+1}_{g,n,1}}(\Gamma,\widetilde{\alpha}) && \rTo^{\delta'''}    &  & \bigoplus_{\Gamma\in Rib^{k+1}_{g,n,1}}(\Gamma,\alpha)  &   \\
\dTo^{} & & & & \dTo^{} & &     &  \\
\vdots & &  & &  \vdots  &  \\
\end{diagram}
\end{equation*}
Here morphism $\delta'''$ acts on the orientation space by the following rule: on a component corresponding to the ordering of vertices and half-edges it is identical, and on $\det(B(\Gamma))$ it acts by \eqref{loc1}. Since, in each homological degree, a morphism $\delta'''$ is an isomorphism, the desired result follows.

\end{proof}

\begin{remark} Alternatively following the scheme of the proof of Lemma $3.0.7$ from \cite{Cost} one can prove Lemma \ref{B0} by induction. To do that, one has to consider the version $K_{g,n,l}^L$ of the orbispace $K_{g,n}^L$ where we allow marked points on the boundary. Hence we consider the morphism:
$$
D_{g,n,l}\hookrightarrow K_{g,n,l}^L
$$
Further, we apply the construction from \textit{ibid.} and get a morphism between simplicial orbispaces, with simplicies, given by products of the same orbispaces of lower dimension. Passing to the total space, we get the induction step.
\end{remark}

We start with the following:
\begin{Lemma}\label{fib} For every $g\geq 0$ and $n\geq 1$ such that $2g+n-2>0$ the morphism:
$$\pi_{\overline{\mathcal N}_{g,n,0,1}}\colon \overline{\mathcal N}_{g,n,0,1}\longrightarrow \overline{\mathcal N}_{g,n}$$
is a smooth and proper oriented fibration of orbifolds with corners.

\end{Lemma} 
\begin{proof} The fact that it is a proper morphism is obvious, it is a pullback of proper morphism \eqref{forcc}.  Let us show that this is fibration (in the orbifold sense). This is completely parallel to the proof of Lemma $3.0.5$ from \cite{Cost}. The marked points are allowed to vary along the surface, including a boundary (in that case, we blow a marked disk attached at this point or separate a node by a marked disk). Further the morphism $\pi_{\overline{\mathcal N}_{g,n,0,1}}$ is locally modelled as a projection with fibers isomorphic to $\overline{\mathcal H},$ where $\overline{\mathcal H}:=\{z\in \mathbb C\,| \mathrm{Im}(z)\geq 0\}$ is the closed upper complex half-plane. Hence the morphism $\pi_{\overline{\mathcal N}_{g,n,0,1}}$ is an oriented fibration. 
\end{proof} 
For every $g\geq 0$ and $n\geq 1$ we denote by $\mathcal N_{g,n}^L$ the locus of Liu's moduli space $\overline{\mathcal N}_{g,n}^L$ which consists of stable bordered surfaces without complex nodes (we allow only nodes of type (a) and (b)). $\mathcal N_{g,n}^L$ is the open suborbifold with corners of dimension $6g-6+3n,$ with interior $\mathcal N_{g,n}.$ We have a morphism:
 $$
\xi_{ \overline{\mathcal N}_{g,n,0,1}}\colon \overline{\mathcal N}_{g,n,0,1}\longrightarrow {\mathcal N}_{g,n+1}^L,
 $$
which is defined by replacing the unique marked point in the interior of the surface with the cusp. Denote by $j\colon \overline{\mathcal N}_{g,n,0,1}^{\geq 3}\hookrightarrow \overline{\mathcal N}_{g,n,0,1}$ the locus of stable bordered surfaces where a component with a marked point has valence greater than $2.$ Applying Lemma \ref{KS} and Lemma \ref{fib} we get that:
$$
\pi_{\overline{\mathcal N}_{g,n,0,1}}^!\epsilon_n\cong \pi_{\overline{\mathcal N}_{g,n,0,1}}^*\epsilon_n\otimes j_!\underline{\mathbb Q}_{\overline{\mathcal N}_{g,n,0,1}^{\geq 3}}[2].
$$
Applying the projection formula and using isomorphism of local system \eqref{loc1} we get the morphism of sheaves:
\begin{equation}\label{loc2} 
\psi\colon \pi^!_{\overline{\mathcal N}_{g,n,0,1}}\epsilon_{n}[-2]\overset{}{\longrightarrow}\xi_{\overline{\mathcal N}_{g,n,0,1}}^*\epsilon_{n+1}
\end{equation}
We have the following:
\begin{Lemma}\label{mor} Morphism \eqref{loc2} induces the quasi-isomorphism:
$$
\psi\colon C_{\hdotc}(\overline{\mathcal N}_{g,n,0,1},\xi_{\overline{\mathcal N}_{g,n,0,1}}^*\epsilon_{n+1})\overset{\sim}{\longrightarrow} C_{\hdotc-2}(\overline{\mathcal N}_{g,n,0,1},\pi^!_{\overline{\mathcal N}_{g,n,0,1}}\epsilon_{n})
$$

\end{Lemma}

\begin{proof} Morphism \eqref{loc2} can be naturally factored:
$$
\pi^!_{\overline{\mathcal N}_{g,n,0,1}}\epsilon_{n}[-2]\overset{\sim}{\longrightarrow}j_!j^*\xi_{\overline{\mathcal N}_{g,n,0,1}}^*\epsilon_{n+1}\longrightarrow j_*j^*\xi_{\overline{\mathcal N}_{g,n,0,1}}^*\epsilon_{n+1}
$$
Applying the Gysin triangle to prove this Lemma it is enough to show that:
$$H_{\hdotc}(\overline{\mathcal N}_{g,n,0,1}^{\leq 2},\xi_{\overline{\mathcal N}_{g,n,0,1}}^*\epsilon_{n+1})=0,$$
 where $\overline{\mathcal N}_{g,n,0,1}^{\leq 2}$ is the closed complement to the morphism $j.$ Denote by $D_{g,n,0,1}^{\leq 2}$ the corresponding locus of nodal disks i.e. 
 $$\overline{\mathcal N}_{g,n,0,1}^{\leq 2}\times_{\overline{\mathcal N}_{g,n,0,1}} D_{g,n,0,1}:=D_{g,n,0,1}^{\leq 2}.$$
 Following Lemma $3.0.7$ from \cite{Cost} one can prove that the natural inclusion $D_{g,n,0,1}^{\leq 2}\hookrightarrow \overline{\mathcal N}_{g,n,0,1}^{\leq 2}$ is the weak equivalence. Indeed both spaces is a product of similar moduli spaces of lower dimension. We leave details to the reader. Finally one can explicitly compute that $H_{\hdotc}(D_{g,n,0,1}^{\leq 2},\xi_{\overline{\mathcal N}_{g,n,0,1}}^*\epsilon_{n+1})=0,$ indeed the Kontsevich-Penner differential $\delta$ admits the decomposition $\delta=\delta_{\geq 3}+\delta_{\leq 2},$ where $\delta_{\leq 2}$ creates one valent vertex with a marking and  $\delta_{\geq 3}$ does not. From this it is easy to see that all classes vanishes.

 \end{proof} 
Consider the fibered square:

\begin{equation*}
\begin{diagram}[height=2.3em,width=2.3em]
D_{g,n,0,1}  & & \rTo_{\sim}^{} &  &  \overline{\mathcal N}_{g,n,0,1}\\
\dTo_{\pi_{D_{g,n,0,1}}}^{} & & & &  \dTo_ {\pi_{\overline{\mathcal N}_{g,n,0,1}}}^{}  \\
D_{g,n} && \rTo^{u}_{\sim}  &&  \overline{\mathcal N }_{g,n}   \\
\end{diagram}
\end{equation*}
The morphism $\pi_{\overline{\mathcal N}_{g,n,0,1}}$ is transversal to the closed inclusion of  $D_{g,n}.$ Applying Lemma \ref{fib} and using the isomorphism from Lemma \ref{bbm} we have a morphism of sheaves:
\begin{equation}\label{loc4} 
\psi\colon \pi^!_{D_{g,n,0,1}}\epsilon_{n}[-2]\overset{}{\longrightarrow} \xi_{D_{g,n,0,1}}^*\epsilon_{n+1},
\end{equation}
which is induced by \ref{loc4}. By Lemma \ref{mor} the latter map induces the quasi-isomorphism of the corresponding chain complexes. Hence we give:

\begin{Def} We define the \textit{geometric Bridgeland differential} (as the morphism in the derived category of vector spaces):
$$
\Delta_1^{bor}\colon C_{\hdotc}(D_{g,n},\epsilon_n) \longrightarrow C_{\hdotc+2}(D_{g,n+1},\epsilon_{n+1})
$$
By the rule:
\begin{equation*}\label{}
\begin{diagram}[height=3.7em,width=4.3em]
&& C_{\hdotc+2}(D_{g,n,0,1},\xi_{D_{g,n,0,1}}^*\epsilon_{n+1})   && \\
& \ldTo_{\sim}^{\psi} & & \rdTo^{\xi_{D_{g,n,0,1}*}} & \\ 
 C_{\hdotc}(D_{g,n,0,1},\pi^!_{D_{g,n,0,1}}\epsilon_n)   &  &    & & C_{\hdotc+2}(K_{g,n+1}^L,\epsilon_{n+1}) &  \\
\uTo_{\pi^!_{D_{g,n,0,1}}}^{}  & &  &  &  \uTo_ {\sim}^{n_*}  && \\
 C_{\hdotc}(D_{g,n},\epsilon_n) &    & \rDotsto^{\Delta_1^{bor}} &   &  C_{\hdotc+2}(D_{g,n+1},\epsilon_{n+1})    \\
\end{diagram}
\end{equation*}
\end{Def}

\begin{Th}\label{B} For $g\geq 0,n\geq 1$ such that $2g+n-2>0,$ the following diagram commutes (in the derived category of vector spaces):

\begin{equation*}
\begin{diagram}[height=3.5em,width=4.2em]
 \prod_{g,n}^{\infty} C_{\dim \mathcal N_{g,n}-2dg-\hdotc}(D_{g,n},\epsilon_n)  & & \rTo_{}^{\Delta_1^{bor}} &  &   \prod_{g,n}^{\infty} C_{\dim \mathcal N_{g,n}+2-2dg-\hdotc}(D_{g,n+1},\epsilon_{n+1}) &  \\
\uTo_{\sim}^{} & & & &  \uTo_ {\sim}^{}  \\
\textsf{RGC}^{\hdot}_d(\delta)_{\geq 3} & &  \rTo_{}^{\Delta_1} &&   \textsf{RGC}^{\hdot+1}_d(\delta)_{\geq 3}   \\
\end{diagram}
\end{equation*}
\end{Th}

\begin{proof} 
Consider the following locus $\overline{\mathcal N}_{g,n,2}'$ in $\overline{\mathcal N}_{g,n,2}$ which consists of all stable bordered surfaces with markings lying on the same boundary. We have the corresponding forgetful morphism $\pi_{\overline{\mathcal N}_{g,n,2}'}\colon \overline{\mathcal N}_{g,n,2}'\longrightarrow \overline{\mathcal N}_{g,n},$ which is a proper and smooth fibration of orbifolds with corners (see Lemma $3.0.5$ in \cite{Cost}). The relative dualising sheaf $\pi_{\overline{\mathcal N}_{g,n,2}'}^!\underline{\mathbb Q}_{\overline{\mathcal N}_{g,n}}$ has the following description:
\begin{equation}\label{loc6}
\pi_{\overline{\mathcal N}_{g,n,2}'}^!\underline{\mathbb Q}_{\overline{\mathcal N}_{g,n}}\cong l_!\epsilon_2[2],
\end{equation} 
where $l\colon \overline{\mathcal N}_{g,n,2}^{\leq 3}\hookrightarrow \overline{\mathcal N}_{g,n,2}'$ is the locus of stable bordered surfaces, where irreducible components with markings have valency greater than two. 
$\epsilon_2$ is a sign local system acting on marked points on a boundary. We also have the corresponding clutching morphism $\xi_{\overline{\mathcal N}_{g,n,2}^2}\colon \overline{\mathcal N}_{g,n,2}'\longrightarrow \overline{\mathcal N}_{g,n+1}$ which is proper.  By $D_{g,n,2}'$ we denote the pullback along the inclusion of $D_{g,n}.$ The  corresponding clutching morphism $\xi_{D_{g,n,2}'}\colon D_{g,n,2}'\longrightarrow D_{g,n+1}$ will be proper. We denote by $\pi_{D_{g,n,2}'}\colon D_{g,n,2}'\longrightarrow D_{g,n}$ the corresponding forgetful morphism. Note that $\pi_{\overline{\mathcal N}_{g,n,2}'}$ is transversal to the inclusion of $D_{g.n}$ Analogous to \eqref{loc2} we have the morphism of local systems:
\begin{equation}\label{loc3}
\theta\colon \pi_{\overline{\mathcal N}_{g,n,2}'}^!\epsilon_n [-2]\longrightarrow \xi_{\overline{\mathcal N}_{g,n,2}'}^*\epsilon_{n+1},
\end{equation} 
defined by the following rule:
$$
\rho\colon b'\wedge\dots\wedge b\wedge\dots \wedge b''\otimes m_{b}\wedge m_{b}\longmapsto b'\wedge\dots\wedge m_{b}\wedge m_{b}\wedge\dots \wedge b'',$$
where we place new boundaries instead of a boundary $b$ (the latter boundary carries marked points). 
The morphism $\theta$ induces the quasi-isomorphism of the corresponding chain complex (cf. Lemma \ref{mor}). Hence we define the following pull-push zigzag morphism:
\begin{equation*}\label{}
\begin{diagram}[height=3.5em,width=4.2em]
 C_{\hdotc}^{}(D_{g,n,2}',\pi_{D_{g,n,2}'}^!\epsilon_{n})  &  & \lTo_{\sim}^{\theta}   & & C_{\hdotc+2}^{}(D_{g,n,2}',\xi_{D_{g,n,2}'}^*\epsilon_{n+1}) &  \\
\uTo_{\pi^!_{D_{g,n,2}'}}^{}  & &  &  &  \dTo_{}^{\xi_{D_{g,n,2}'*}}  && \\
 C_{\hdotc}(D_{g,n},\epsilon_n) &    & \rDotsto^{\Delta^{bor'}_1} &   &  C_{\hdotc+2}(D_{g,n+1},\epsilon_{n+1})    \\
\end{diagram}
\end{equation*}
We have the following:
\begin{Prop}\label{KeyL3} For any $g\geq 0$ and $n\geq 1$ such that $2g+n-2->0$ the morphism: 
$$
\Delta_1^{bor'}\colon C_{\hdotc}(D_{g,n},\epsilon_n) \longrightarrow C_{\hdotc+2}(D_{g,n+1},\epsilon_{n+1}) 
$$
vanishes.

\end{Prop} 

\begin{proof} 

Consider the following diagram:
\begin{equation*}
\begin{diagram}[height=3.5em,width=4.4em]
  {\mathcal N}_{g,n,2}' & \rTo^{{\widetilde{v}}_{}}_{\sim}    & \overline{\mathcal N}_{g,n,2}' & \lTo^{{\widetilde{u}}_{}}_{\sim }      &   D_{g,n,2}' &\\
 \dTo^{\pi_{{\mathcal N}_{g,n,2}'}}    &  & \dTo^{\pi_{\overline{\mathcal N}_{g,n,2}'}}  &  & \dTo^{\pi_{D_{g,n,2}'}}  &  \\
{\mathcal N}_{g,n} & \rTo^{v}_{\sim} & \overline{\mathcal N}_{g,n}  & \lTo^{u}_{\sim} &D_{g,n}
\end{diagram}
\end{equation*}
By $\mathcal N_{g,n,2}',$ we denote the pullback of $\overline{\mathcal N}_{g,n,2}'$ along the inclusion of the smooth locus. Morphisms $\widetilde{u}$ and $\widetilde{v}$ are weak homotopy equivalences since $\pi_{\overline{\mathcal N}_{g,n,2}'}$ is a fibration. 
Consider the canonical projection morphism:
\begin{equation}\label{spr}
\mathcal N_{g,n,[2]}'\longrightarrow \mathcal N_{g,n,2}'
\end{equation}
This morphism is a $\Sigma_2$-covering, and hence the corresponding pushforward induces the quasi-isomorphism:
$$
(C_{\hdotc}(\mathcal N_{g,n,[2]}',\epsilon_n)\otimes \mathrm{sgn}_2)^{\Sigma_2}\overset{\sim}{\longrightarrow} C_{\hdotc}(\mathcal N_{g,n,2}',\epsilon_n\otimes \epsilon_2)
$$
We claim that the symmetric group $\Sigma_2$ acts identically on $H_{\hdotc}(\mathcal N_{g,n,2}',\epsilon_n).$ Consider the following fibration:
$$
FM^{[2]}(S^1)\longrightarrow {\mathcal N}_{g,n,[2]}'\longrightarrow \mathcal N_{g,n}.
$$
Where $FM^{[2]}(S^1)$ is the Fulton-MacPherson space of two labelled points on a circle i.e. compactification of the configuration spaces $Conf^{[2]}(S^1).$ Morphism \eqref{spr} induces the following commutative diagram of fibrations:
\begin{equation*}\label{}
\begin{diagram}[height=2.4em,width=4.2em]
FM^{2}(S^1) &  \rTo  & {\mathcal N}_{g,n,2}' & \rTo  &  \mathcal N_{g,n} \\
\uTo_{p_1}^{}  & & \uTo_{p_2}  &  &  \uTo_{p_3}^{} \\
FM^{[2]}(S^1) &  \rTo  & {\mathcal N}_{g,n,[2]}' & \rTo  &  \mathcal N_{g,n}\\
\end{diagram}
\end{equation*}
The morphism $p_1$ is a map of degree two and hence induces an isomorphism in the rational homology, and the morphism $p_3$ is identity. Hence, the pushforward long the morphism $p_2$ induces an isomorphism over $\mathbb Q.$ Hence, the action is trivial. 
\end{proof}

\begin{Prop}\label{KeyL4} For any $g\geq 0$ and $n\geq 1$ such that $2g+n-2>0$ the morphism:
$$
\Delta_1^{bor'}\colon C_{\hdotc}(D_{g,n},\epsilon_n) \longrightarrow C_{\hdotc+2}(D_{g,n+1},\epsilon_{n+1}) 
$$
admits the following decomposition:
$$
\Delta_1^{bor'}=\Delta_1+\Delta_1^{bor}.
$$

\end{Prop}

\begin{proof}

Let $(\Gamma,\alpha)$ be a cochain in $C_{\hdotc}(D_{g,n},\epsilon_n)$ i.e. a ribbon graph of genus $g$ with $n$ boundaries together with an element $\alpha$ in $\det(E(\Gamma)).$ We suppose that the number of edges equals to $l.$ The corresponding stratum in $D_{g,n}$ will be denoted by $X(\Gamma).$ By Proposition \ref{Gor1} and we have the following description of the Gysin pullback:
$$
\pi^!_{D_{g,n,2}'}(\Gamma,\alpha)=H_{\dim \pi^{-1}_{D_{g,n,2}'}(X({\Gamma}))}^{BM}(\pi^{-1}_{D_{g,n,2}'}(X({\Gamma})),\pi^!_{D_{g,n,2}'}\epsilon_{n}) 
$$
Using isomorphism \eqref{loc6} we get that the homology above is given by the following collection:
$$
\pi^!_{D_{g,n,2}'}(\Gamma,\alpha)=\sum_{b\in B(\Gamma)} \sum_{c,c'\in C(b)} (\widetilde{e}_{c,c'}(\Gamma),\alpha'), 
$$
where $\widetilde{e}_{c,c'}$ attaches two non-labelled hairs to the ribbon $\Gamma$ at corners $c$ and $c'.$ $\alpha'$ is an element of the following space:
$$
 \det(V(\Gamma))\otimes \det (H(\Gamma))\otimes \det(B(\Gamma))\otimes \det(M(\Gamma)),
$$
where $M(\Gamma)$ is a set of hairs of the ribbon graph. Assume that $\Gamma$ represents some cohomology class. Then the cohomology class of $\Delta_1^{bor'}(\Gamma)$ will be represented by:
$$\Delta_1^{bor'}(\Gamma)=\sum_{b\in B(\Gamma)} \sum_{c,c'\in C(b)} {e}_{c,c'}(\Gamma)$$
\par\medskip 
Analogously, the cohomology class of $\Delta_1^{bor}(\Gamma)$ can be represented by:
$$\Delta_1^{bor}(\Gamma)=\sum_{b\in B(\Gamma)} \sum_{c=c'\in C(b)} ({e}_{c,c'}(\Gamma),\alpha'').$$
This follows from the geometry of the stratified space $K_{g,n}^L.$ Consider a differential $D$ in the complex $C_{\hdotc}(K_{g,n}^L,\epsilon_n).$ This differential has the following decomposition:
$$
D=\delta+\delta'+\delta''+\delta'''+\delta''''.
$$
Where we have:
\begin{enumerate}[(a)]
\par\medskip 
\item $\delta$ is the standard differential in the Kontsevich-Penner ribbon graph complex, 
\item $\delta'$ is the differential which is defined by splitting a vertex on nodal surfaces which has a "genus one" component, 
\item $\delta''$ is defined by splitting a vertex on surfaces with a marking on an irreducible component, 
\item $\delta'''$ is defined by contracting a loop over a "non-singular" boundary of a nodal surface with a "genus one" component,
\item $\delta''''$ is defined by contracting an arc with endpoints on distinct boundary components of a genus one irreducible component. 
\par\medskip 
\end{enumerate}
From Lemma \ref{B0} we get that for any class $\omega$ of $D_{g,n,0,1}$ there exists a class $\omega'$ of $T_{g,n+1}$ (in the Borel-Moore homology). From the description above we get that that $D(\omega')=\delta'''(\omega) +\delta''''(\omega).$ Hence we get that chains are homologous. Comparing with \eqref{B1} we get the decomposition.

\end{proof} 
Hence by Proposition \ref{KeyL3} we obtain that the differential $\Delta_1^{bor}$ coincides with $\Delta_1.$ 

\end{proof}

\subsection{T. Willwacher's differential $\nabla_1^{bor}$}

For every $g\geq 0$ and $n\geq 1$ we have a sequence of inclusions:
$$
\mathcal N_{g,n}\overset{u}{\hookrightarrow} \overline{\mathcal N}_{g,n}\overset{q}{\hookrightarrow} {\mathcal N}_{g,n}^L
$$
 Note that since $\mathcal N_{g,n}$ is an interior of ${\mathcal N}_{g,n}^L$ the morphism $q\circ u$ is a weak homotopy equivalence. Hence by two-out-of-three property of weak equivalences, the morphism $q$ is a homotopy equivalence as well. We have a morphism 
$$
\nabla_1^{bor'}\colon C_{\hdotc}(\overline{\mathcal N}_{g,n},\epsilon_n)  \longrightarrow C_{\hdotc+2}(\overline{\mathcal N}_{g,n+1},\epsilon_{n+1})
$$
defined by the rule:
\begin{equation*}\label{}
\begin{diagram}[height=3.3em,width=4.3em]
&& C_{\hdotc+2}(\overline{\mathcal N}_{g,n,0,1},\xi^*_{\overline{\mathcal N}_{g,n,0,1}}\epsilon_{n+1})   && \\
& \ldTo_{\sim}^{\psi} & & \rdTo^{\xi_{\overline{\mathcal N}_{g,n,0,1}*}} & \\ 
  C_{\hdotc}(\overline{\mathcal N}_{g,n,0,1},\pi^!_{\overline{\mathcal N}_{g,n,0,1}}\epsilon_n)  &  &    & & C_{\hdotc+2}({\mathcal N}_{g,n+1}^L,\epsilon_{n+1}) &  \\
\uTo_{\pi^!_{\overline{\mathcal N}_{g,n,0,1}}}^{}  & &  &  &  \uTo_ {\sim}^{q_*}  && \\
 C_{\hdotc}(\overline{\mathcal N}_{g,n},\epsilon_n)  &    & \rDotsto^{\nabla_1^{bor'}} &   &   C_{\hdotc+2}(\overline{\mathcal N}_{g,n+1},\epsilon_{n+1})    \\
\end{diagram}
\end{equation*}
Hence we give the following:
\begin{Def} For every $g\geq 0$ and $n\geq 1$ such that $2g+n-2>0$ by ${\nabla}_1^{bor}$ we denote the unique morphism:
$$
\nabla_1^{bor}\colon C_{\hdotc}(\mathcal N_{g,n},\epsilon_n)  \longrightarrow C_{\hdotc+2}(\mathcal N_{g,n+1},\epsilon_n),
$$
which suits in the following diagram:
defined by the rule:
\begin{equation*}\label{fb1}
\begin{diagram}[height=3em,width=4.2em]
C_{\hdotc}(\overline{\mathcal N}_{g,n},\epsilon_{n}) & &  \rTo^{\nabla_1^{bor'}} &  &   C_{\hdotc+2}(\overline{\mathcal N}_{g,n+1},\epsilon_{n+1})\\
\uTo_{v_*}^{\sim} & & &  & \uTo_{v_*}^{\sim}  & \\
C_{\hdotc}(\mathcal N_{g,n},\epsilon_n)   & &  \rTo^{\nabla_1^{bor}} &  & C_{\hdotc+2}(\mathcal N_{g,n+1},\epsilon_{n+1})   \\
\end{diagram}
\end{equation*}
\end{Def}

We have the following:

\begin{Prop}\label{comp} For every $g\geq0$ and $n\geq 1$ such that $2g+n-2>0$ the following square commutes (in the derived category of vector spaces):

\begin{equation*}
\begin{diagram}[height=3.5em,width=4.2em]
\prod_{g,n}^{\infty} C_{\dim \mathcal N_{g,n}-2dg-\hdotc}(\mathcal N_{g,n},\epsilon_n) & &  \rTo_{}^{{\nabla}_1^{bor}} &  &   \prod_{g,n}^{\infty}C_{\dim \mathcal N_{g,n+1}-2dg-\hdotc}(\mathcal N_{g,n+1},\epsilon_{n+1}) &  \\
\uTo_{\sim}^{v_*^{-1}\circ u_*} & & &  & \uTo_ {\sim}^{v_*^{-1}\circ u_*}  && \\
\textsf{RGC}_{d}^{\hdot}(\delta)_{\geq 3} & &  \rTo_{}^{\Delta_1} &  & \textsf{RGC}_{d}^{\hdot+1}(\delta)_{\geq 3}   \\
\end{diagram}
\end{equation*}
\end{Prop}

\begin{proof} We have the following diagram: 
$$
\begin{diagram}[height=3em,width=8.2em]
 C_{\hdotc+2}(\overline{\mathcal N}_{g,n+1},\epsilon_{n+1}) & \lTo^{{u}_*}_{\sim}   &   C_{\hdotc+2}(D_{g,n+1},\epsilon_{n+1})\\
  \dTo^{q_*}_{\sim} & & \dTo^{n_*}_{\sim}  & \\
 C_{\hdotc+2}({\mathcal N}_{g,n+1}^L,\epsilon_{n+1}) & \lTo^{\widetilde{u}_*}_{\sim}  &   C_{\hdotc+2}(K_{g,n+1}^L,\epsilon_{n+1})\\
  \uTo^{\xi_{\overline{\mathcal N}_{g,n,0,1}*}} & & \uTo_{\xi_{D_{g,n,0,1}*}}  & \\
C_{\hdotc+2}(\overline{\mathcal N}_{g,n,0,1},\xi_{\overline{\mathcal N}_{g,n,0,1}}^*\epsilon_{n+1}) & \lTo^{{u}_*}_{}   &   C_{\hdotc+2}(D_{g,n,0,1},\xi^*_{D_{g,n,0,1}}\epsilon_{n+1})  \\
 \uTo^{\psi}_{\sim} & & \uTo^{\psi}_{\sim}  & \\
C_{\hdotc}(\overline{\mathcal N}_{g,n,0,1},{\pi}_{\overline{\mathcal N}_{g,n,0,1}}^!\epsilon_n)& \lTo^{{u}_*}_{}   &   C_{\hdotc}(D_{g,n,0,1},\pi^!_{D_{g,n,0,1}}\epsilon_n)  \\
  \uTo^{{\pi}_{\overline{\mathcal N}_{g,n,0,1}}^!} & & \uTo^{{\pi}_{D_{g,n,0,1}}^!}  & \\
C_{\hdotc}(\overline{\mathcal N}_{g,n},\epsilon_n) & \lTo^{u_*}_{\sim} & C_{\hdotc}(D_{g,n},\epsilon_n)   \\
\end{diagram}
$$
Hence by Proposition \ref{bbm} the bottom square of the diagram above commutes. Moreover $\widetilde{u}_*$ is a quasi-isomorphism (by two-out-of-three property of weak equivalences). Hence we get that that the diagram above is commutative. Applying Theorem \ref{B} we get the result. 

\end{proof}

\begin{Th}\label{comml} For every $g\geq0$ and $n\geq 1,$ such that $2g+n-2>0$ the following square commutes (in the derived category of vector spaces):
\begin{equation*}
\begin{diagram}[height=3.2em,width=4em]
 C_c^{\hdot} (\mathcal M_{g,n}/\Sigma_{n},\epsilon) & & \rTo^{\nabla_1}  &  &   C_c^{\hdot} (\mathcal M_{g,n+1}/\Sigma_{n+1},\epsilon) &\\
 \dTo^{\rho_{\mathcal M_{g,n}/\Sigma_{n}}} & &  &  & \dTo_{\rho_{\mathcal M_{g,n+1}/\Sigma_{n+1}}}  &    \\
C_c^{\hdot+n}(\mathcal N_{g,n},\mathbb Q) &  & \rTo^{\mathbb D\circ \nabla_1^{bor}\circ \mathbb D} & & C_c^{\hdot+n+1}(\mathcal N_{g,n+1},\mathbb Q)
\end{diagram}
\end{equation*}

\end{Th}
\begin{proof} For every $g\geq 0, n\geq 1,$ such that $2g+n-2>0$ we denote by ${\mathcal N}_{g,n}^{L\,rt}$ the open locus in Liu's moduli space which consists of stable bordered surfaces $\Sigma$ of the following type:
\begin{enumerate}[(a)]
\item $\Sigma$ does not have complex nodes, then $\Sigma\in \mathcal N_{g,n}^L,$
\item $\Sigma$ has complex singularities, then $\Sigma$ is a stable surface with boundaries being cusps and only "rational tails" as complex singularities i.e. nodal surfaces $\Sigma\in \overline{\mathcal N}_{g,n}^L$ such that the normalisation of complex nodes of $\Sigma$ consists of the unique stable bordered surfaces of genus $g$ and stable bordered surfaces of genus $0.$ 
\end{enumerate}
 ${\mathcal N}_{g,n}^{L\,rt}$ is an orbifold with corners of dimension $6g-6+3n.$ We will also use a notation ${\mathcal N}_{g,n,0,1}^{L\,rt}$ for the pullback of $\widetilde{\pi}^{comp}\colon \widetilde{\mathcal N}_{g,n,0,1}^L \longrightarrow \overline{\mathcal N}_{g,n}^L$ along the canonical inclusion ${\mathcal N}_{g,n}^{L\,rt}\longrightarrow \overline{\mathcal N}_{g,n}^L.$ Note that by the construction we have a proper morphism:
\begin{equation}\label{mu}
\pi_{{\mathcal N}_{g,n,0,1}^{L\,rt}}\colon{\mathcal N}_{g,n,0,1}^{L\,rt}\longrightarrow {\mathcal N}_{g,n}^{L\,rt}
\end{equation} 
This morphism is a smooth and oriented fibration of stratified manifolds with boundary (cf. Lemma \ref{fib}). Following \cite{AK3} we will use a notation $\mathcal M_{\bullet_{g,n}}^{rt}$ for the moduli stack $\mathcal M_{g,n}^{rt}/\Sigma_n,$ where $\mathcal M_{g,n}^{rt}\subset \overline{\mathcal M}_{g,n}$ is a moduli stack of stable curves with rational tails. We have an inclusion defined by replacing marked points with cusps:
\begin{equation}\label{in1}
c\colon \mathcal M_{\bullet_{g,n}}^{rt}\longrightarrow {\mathcal N}_{g,n}^{L\,rt}
\end{equation}
Note that morphism \eqref{mu} is transversal to \eqref{in1} in the sense of stratified spaces (cf. Subsection $5$ \cite{Gore}). Consider the pullback of morphism \eqref{mu} along the inclusion above. We have an equivalence of topological stacks:
\begin{equation}
{\mathcal N}_{g,n,0,1}^{L\,rt}\times_{{\mathcal N}_{g,n}^{L\,rt}}\mathcal M_{\bullet_{g,n}}^{rt}\cong \mathcal M_{\bullet_{g,n},\bullet_{g,n+1}}^{rt}.
\end{equation} 
Where by $\mathcal M_{\bullet_{g,n},\bullet_{g,n+1}}^{rt}$ we have denoted the following stack $(\mathcal M_{g,n+1}^{rt}\times [n+1])/ \Sigma_{n+1}.$ Following \cite{AK3} the pullback of morphism \eqref{mu} along inclusion \eqref{in1} will be denoted by:
$$
\widetilde{\mu}_{\bullet_{g,n},\bullet_{g,n+1}}\colon \mathcal M_{\bullet_{g,n},\bullet_{g,n+1}}^{rt}\longrightarrow \mathcal M_{\bullet_{g,n}}^{rt}
$$
Note that we also have a morphism:
$$
\xi_{{\mathcal N}_{g,n,0,1}^{L\,rt}}\colon {\mathcal N}_{g,n,0,1}^{L\,rt}\longrightarrow {\mathcal N}_{g,n+1}^{L\,rt},
$$
which is defined by replacing the unique complex marked point with a cusp. We have an obvious commutativity property:
\begin{equation*}
\begin{diagram}[height=3em,width=3.4em]
 \mathcal M_{\bullet_{g,n},\bullet_{g,n+1}}^{rt} & & \rTo^{\widetilde{c}} &  &   {\mathcal N}_{g,n,0,1}^{L\,rt}\\
 \dTo^{ \pi_{\bullet_{g,n},\bullet_{g,n+1}}}_{} & &  &  & \dTo_{\xi_{{\mathcal N}_{g,n,0,1}^{L\,rt}}}  \\
\mathcal M_{\bullet_{g,n+1}}^{rt} &  & \rTo^{c} & & {\mathcal N}_{g,n+1}^{L\,rt}.\\
\end{diagram}
\end{equation*}
Where by $\pi_{\bullet_{g,n},\bullet_{g,n+1}}\colon \mathcal M_{\bullet_{g,n},\bullet_{g,n+1}}^{rt}\longrightarrow \mathcal M_{\bullet_{g,n+1}}^{rt}$ we have denoted the canonical projection. We have an isomorphism of sheaves:  
$$\psi\colon \pi_{{\mathcal N}_{g,n,0,1}^{L\,rt}}^!\epsilon_n[-2]\overset{\sim}{\longrightarrow}\xi_{{\mathcal N}_{g,n,0,1}^{L\,rt}}^*\epsilon_{n+1}$$ 
This morphism is defined by \eqref{loc2}.  Applying Proposition \ref{bbm} we can organise all data above into the commutative diagram:
\begin{equation*}
\begin{diagram}[height=3.5em,width=3.4em]
 C_{\hdotc+2} (\mathcal M_{\bullet_{g,n+1}}^{rt},\epsilon_{n+1}) & & \rTo^{c_{*}}  &  &   C_{\hdotc+2} ({\mathcal N}_{g,n+1}^{L\,rt},\epsilon_{n+1}) &\\
 \uTo^{\pi_{\bullet_{g,n},\bullet_{g,n+1}*}} & &  &  & \uTo_{\xi_{\mathcal N_{g,n,0,1}^{L\,rt}*}}  &    \\
 C_{\hdotc+2} (\mathcal M_{\bullet_{g,n},\bullet_{g,n+1}}^{rt},\pi_{\bullet_{g,n},\bullet_{g,n+1}}^*\epsilon_{n+1}) & & \rTo^{\widetilde c_*}  &  &   C_{\hdotc+2} ({\mathcal N}_{g,n,0,1}^{L\,rt},\xi_{{\mathcal N}_{g,n,0,1}^{L\,rt}}^*\epsilon_{n+1}) &\\
 \uTo^{\widetilde{\mu}_{\bullet_{g,n},\bullet_{g,n+1}}^!} & &  &  & \uTo_{\pi_{{\mathcal N}_{g,n,0,1}^{L\,rt}}^!}  &    \\
C_{\hdotc}(\mathcal M_{\bullet_{g,n}}^{rt},\epsilon_{n}) &  & \rTo^{c_{*}} & & C_{\hdotc}({\mathcal N}_{g,n}^{L\,rt},\epsilon_n)
\end{diagram}
\end{equation*}
We have the open inclusion $\tilde{j}^{rt}\colon \widetilde{\mathcal N}_{g,n}\hookrightarrow {\mathcal N}_{g,n}^{L\,rt}.$ We have the following quasi-isomorphism:
\begin{equation}\label{rt}
\tilde{j}^{rt}_*\colon C_{\hdotc} (\widetilde{\mathcal N}_{g,n},\epsilon_n) \longrightarrow C_{\hdotc} ({\mathcal N}_{g,n}^{L\,rt},\epsilon_n)
\end{equation} 
 The morphism $\tilde{j}^{rt}$ can be naturally factored: 
$$\widetilde{\mathcal N}_{g,n}\hookrightarrow \mathcal N_{g,n}^L \overset{g}{\hookrightarrow} {\mathcal N}_{g,n}^{L\,rt}.$$
 Note that the first and the second arrows are weak homotopy equivalences by two-out-of-three property of weak homotopy equivalences. Hence $\tilde{j}^{rt}$ is a weak-homotopy equivalence. 
We get the following commutative diagram:
\begin{equation*}
\begin{diagram}[height=3.1em,width=3.5em]
 C_{\hdotc} (\mathcal M_{{g,n}}/\Sigma_n,\epsilon_{n}) & & \rTo^{a_{*}}_{\sim} &  &   C_{\hdotc} (\widetilde{\mathcal N}_{g,n},\epsilon_{n}) &\\
 \dTo^{\tilde{j}_{rt*}}_{\sim} & &  &  & \dTo_{\tilde{j}_{rt*}}^{\sim}  &    \\
C_{\hdotc}(\mathcal M_{\bullet_{g,n}}^{rt},\epsilon_{n}) &  & \rTo^{c_{*}} & & C_{\hdotc}({\mathcal N}_{g,n}^{L\,rt},\epsilon_n)
\end{diagram}
\end{equation*}
By two-out-of-three property of homotopy equivalences, the lower arrow of the diagram above is the quasi-isomorphism. We define a morphism:
$$\nabla_1^{bor''}\colon C_{\hdotc}(\widetilde {\mathcal N}_{g,n},\epsilon_n) \longrightarrow C_{\hdotc+2} (\widetilde {\mathcal N}_{g,n+1},\epsilon_{n+1})$$ 
as the unique morphism which suits in the following commutative diagram:
 \begin{equation*}
\begin{diagram}[height=3.2em,width=4.3em]
 C_{\hdotc}(\widetilde {\mathcal N}_{g,n},\epsilon_n) & & \rTo^{\nabla_1^{bor''}} &  &   C_{\hdotc+2} (\widetilde {\mathcal N}_{g,n+1},\epsilon_{n+1})\\
 \dTo^{\tilde j^{rt}_*}_{\sim} & &  &  & \dTo_{\tilde j^{rt}_*}^{\sim}  \\
C_{\hdotc}({\mathcal N}_{g,n}^{L\,rt},\epsilon_n) & & \rTo^{\xi_{\mathcal N_{g,n,0,1}^{L\,rt}*}\pi_{{\mathcal N}_{g,n,0,1}^{L\,rt}}^!} &  &   C_{\hdotc+2} ({\mathcal N}_{g,n+1}^{L\,rt},\epsilon_{n+1})\\
\end{diagram}
\end{equation*} 
Recall that the {Willwacher differential} (Definition $4.4.4$ \cite{AK}) :
$$\nabla_1\colon C^{\hdot}_c(\mathcal M_{{g,n}}/\Sigma_n,\epsilon_{n})\longrightarrow C^{\hdot}_c(\mathcal M_{{g,n+1}}/\Sigma_{n+1},\epsilon_{n+1})$$ 
is defined as the unique morphism which suits in the following commutative diagram:
 \begin{equation*}
\begin{diagram}[height=3.2em,width=4.2em]
 C^{\hdot}_c(\mathcal M_{{g,n}}/\Sigma_{n},\epsilon_{n}) & & \rTo^{\nabla_1} &  &   C^{\hdot}_c(\mathcal M_{{g,n+1}}/\Sigma_{n+1},\epsilon_{n+1})\\
 \dTo^{j^{rt}_!}_{\sim} & &  &  & \dTo_{j^{rt}_!}^{\sim}  \\
C^{\hdot}_c(\mathcal M_{\bullet_{g,n}}^{rt},\epsilon_{n}) &  & \rTo^{\pi_{\bullet_{g,n},\bullet_{g,n+1}!}\widetilde{\mu}_{\bullet_{g,n},\bullet_{g,n+1}}^*} & & C^{\hdot}_c(\mathcal M_{\bullet_{g,n+1}}^{rt},\epsilon_{n+1})
\end{diagram}
\end{equation*} 
Denote by $\nabla_1^{P.V.}$ the corresponding Poincaré-Verdier dual morphism. Hence we get the following:
\begin{Prop}\label{lbw1} For every $g\geq 0$ and $n\geq 1$ such that $2g+n-2>0$ the following diagram commutes (in the derived category of vector spaces):
\begin{equation*}
\begin{diagram}[height=3em,width=3em]
 C_{\hdotc+2} (\mathcal M_{{g,n}}/\Sigma_{n+1},\epsilon_{n+1}) & & \rTo^{a_{*}}_{\sim} &  &   C_{\hdotc+2} (\widetilde {\mathcal N}_{g,n+1},\epsilon_{n+1}) &\\
 \uTo^{\nabla_1^{P.V.}}_{} & &  &  & \uTo_{\nabla_1^{bor''}}^{}  &    \\
C_{\hdotc}(\mathcal M_{{g,n}}/\Sigma_n,\epsilon_{n}) &  & \rTo^{a_{*}}_{\sim} & & C_{\hdotc}(\widetilde {\mathcal N}_{g,n},\epsilon_n)
\end{diagram}
\end{equation*}

\end{Prop} 

Note that by the construction we have an open inclusion $\overline{\mathcal N}_{g,n}\rightarrow {\mathcal N}_{g,n}^{L\,rt},$ which is transversal (in the stratified sense) to \eqref{mu}. The pullback of the moduli space ${\mathcal N}_{g,n,0,1}^{L\,rt}$ along this open inclusion is isomorphic to $\overline{\mathcal N}_{g,n,0,1},$ with the canonical morphism identified with \eqref{forc1}. Applying Proposition \ref{bbm} together with isomorphism \eqref{loc2} we get the following commutative diagram:
\begin{equation*}
\begin{diagram}[height=3.2em,width=2.7em]
 C_{\hdot+2} ({\mathcal N}_{g,n+1}^{L\,rt},\epsilon_{n+1}) & & \lTo^{g_{*}}  &  &   C_{\hdotc+2} ({\mathcal N}_{g,n+1}^{L},\epsilon_{n+1}) &\\
 \uTo^{\xi_{\mathcal N_{g,n,0,1}^{L\,rt}*}} & &  &  & \uTo_{\xi_{\overline{\mathcal N}_{g,n,0,1}*}}  &    \\
 C_{\hdotc+2} ({\mathcal N}_{g,n,0,1}^{L\,rt},\xi_{\mathcal N_{g,n,0,1}^{L\,rt}}^*\epsilon_{n+1}).  & & \lTo^{}  &  &   C_{\hdotc+2} (\overline{\mathcal N}_{g,n,0,1},\xi_{\overline{\mathcal N}_{g,n,0,1}}^*\epsilon_{n+1}) &\\
 \uTo^{\pi_{{\mathcal N}_{g,n,0,1}^{L\,rt}}^!} & &  &  & \uTo_{\pi^!_{\overline{\mathcal N}_{g,n,0,1}}}  &    \\
C_{\hdotc}({\mathcal N}_{g,n}^{L\,rt},\epsilon_n)
 &  & \lTo^{(g\circ q)_*} & & C_{\hdotc}(\overline{\mathcal N}_{g,n},\epsilon_n)
\end{diagram}
\end{equation*}
Hence we have proved the following:
\begin{Prop}\label{lbw2} For every $g\geq 0$ and $n\geq 1$ such that $2g+n-2>0$ the following diagram commutes (in the derived category of vector spaces):
\begin{equation*}
\begin{diagram}[height=3em,width=3em]
C_{\hdotc+2}({\mathcal N}_{g,n+1},\epsilon_{n+1}) & & \rTo^{b_{*}}_{\sim} &  &   C_{\hdotc+2} (\widetilde {\mathcal N}_{g,n+1},\epsilon_{n+1}) &\\
 \uTo^{\nabla_1^{bor}}_{} & &  &  & \uTo_{\nabla_1^{bor''}}  &    \\
C_{\hdotc}({\mathcal N}_{g,n},\epsilon_n) &  & \rTo^{b_{*}}_{\sim} & & C_{\hdotc}(\widetilde {\mathcal N}_{g,n},\epsilon_n)
\end{diagram}
\end{equation*}

\end{Prop} 

Combining Proposition \ref{lbw1} with Proposition \ref{lbw2} and applying Poincaré-Verdier duality we get the result. 
\end{proof}

\begin{Th}[T. Willwacher's conjecture] For every $g\geq0$ and $n\geq 1,$ such that $2g+n-2>0$ the following square commutes (in the derived category of vector spaces):

\begin{equation*}
\begin{diagram}[height=3em,width=4em]
\prod_{g,n}^{\infty} C_c^{\hdot+2dg-n}(\mathcal M_{g,n}/ \Sigma_{n},\epsilon_n) & &  \rTo_{}^{\nabla_1} &  &   \prod_{g,n}^{\infty}C_{c}^{\hdot+2dg-n }(\mathcal M_{g,n+1}/ \Sigma_{n+1},\epsilon_{n+1}) &  \\
\dTo_{\sim}^{\textsf{cost}} & & &  & \dTo_ {\sim}^{\textsf{cost}}  && \\
\textsf{RGC}_{d}(\delta)^{\hdot}_{\geq 3} & &  \rTo_{}^{\Delta_1} &  & \textsf{RGC}_{d}(\delta)^{\hdot+1}_{\geq 3}   \\
\end{diagram}
\end{equation*}
\end{Th}

\begin{proof} Combining Theorem \ref{comml} with Proposition \ref{comp} we obtain the result.

\end{proof} 
\begin{Th}[A. C$\check{\mathrm a}$ld$\check{\mathrm a}$raru's conjecture]\label{CalC} For every $d\in \mathbb Z$ the Merkulov-Willwacher ribbon graph complex $\textsf{RGC}_d^{\hdot}(\delta+\Delta_1)$ is quasi-isomorphic to the totality of the shifted cohomology with compact support of the moduli stacks $\mathcal M_{g}$:
$$
\textsf{cald}\colon \textsf{RGC}_d^{\hdot}(\delta+\Delta_1)\cong_{\geq 2} \prod_{g\geq 2}^{\infty} C^{\hdot+2dg-1}_c(\mathcal M_g,\mathbb Q).
$$
\end{Th}

\begin{proof} For every $g>0$ the genus $g$-part of the Merkulov-Willwacher ribbon graph complex can be identified with the total DG-vector space of the complex:
\begin{equation}
\begin{CD}
C_c^{\hdot}(M_{g,1}^{rib},\mathbb Q) @>{\Delta_1}>> \dots @>{\Delta_1}>> C_c^{\hdot}(M_{g,n}^{rib}/\Sigma_n,\mathbb Q) @>{\Delta_1}>> \dots
\end{CD}
\end{equation}
This follows from decomposition \eqref{sum1} and the fact that the genus of ribbon graphs $R_{k}$ in \eqref{sum2} is always zero. Consider a case when $g=0,$ denote by $\textsf B_0\textsf{RGC}_0^{\hdot}(\delta+\Delta_1)_{\geq 3}$ the genus zero part of the Merkulov-Willwacher ribbon graph complex, which is quasi-isomorphic to the total DG-vector space of the complex:
\begin{equation}
\begin{CD}
\dots @>{\nabla_1}>> \left(C_c^{\hdot}(\mathcal M_{0,n},\mathbb Q)\otimes_{\Sigma_n} \mathrm {sgn}_n\right)^{\Sigma_n} @>{\nabla_1}>> \dots
\end{CD}
\end{equation}
Then it is easy to notice that:\footnote{By Theorem 4.3, Corollary2 from \cite{VAS2}.}
$$
H^{\hdot}(\textsf B_0\textsf{RGC}_0(\delta+\Delta_1)_{\geq 3})=0.
$$
Hence from the result above we get that $H^{\hdot}(\textsf B_0\textsf{RGC}_0(\delta+\Delta_1))$ is controlled by the loop classes above. Thus using the involution\footnote{This involution is defined by replacing the vertices of a ribbon graph with boundaries and vice versa \cite{MW} \cite{CFL}.} on the Merkulov-Willwacher ribbon graph complex we get that any class $R_k$ $k>1$ must has a partner with at least trivalent vertices and $R_1$ must have a partner with one boundary and one edge. Finally applying Theorem $4.4.5$ from \cite{AK} we get the desired quasi-isomorphism.

\end{proof} 

Applying Theorem $4.4.5$ from \cite{AK} we can also compute the cohomology of  $\textsf{RGC}_d^{\hdot}(\delta+\Delta_1)$ in genus $\leq 1:$

\begin{Prop}\label{CalC1} The cohomology of the Merkulov-Willwacher ribbon graph complex has the following description:
$$
H^{\hdot}(\textsf{RGC}_d(\delta+\Delta_1))\cong_{\leq 1} \prod_{n=3}^{\infty} (S_{n+1}\oplus \overline{S}_{n+1}\oplus Eis_{n+1})[2(n+d)-1]\oplus \mathbb Q[2+2d].
$$
Where $S_k$ (resp. $\overline{S}_k$) is a vector space of holomorphic (resp. antiholomorphic) cusp forms of the weight $k$ and $Eis_k$ is a vector space of the Eisenstein series of the weight $k$
\end{Prop}

Denote by $\textsf {GC}_{2d}^{2\, \hdot}(\delta)$ the \textit{Kontsevich graph complex} \cite{Will}. This is a combinatorial cochain complex with cochains given by pairs $(G,or_{G}),$ where $G$ is a graph with at least two valent vertices and no loop edges and $or_{G}\in \det(E(G))$ is an orientation. These elements are considered modulo the relation $(G,or_{G}^{op})=-(G,or_{G}).$ We define the grading of $(G,or_{G})$ by the rule: $$|G|=2d(|V(G)|-1)+(1-2d)|E(G)|.$$ 
The differential $\delta$ is defined by splitting a vertex. Note that the differential $\delta$ preserves a genus of a graph and hence we have a decomposition of the Kontsevich graph complex into subcomplexes by fixing a genus. For $g\geq 1$ $\textsf B_g\textsf {GC}_{2d}^{2\,\hdot}(\delta)$ we denote the genus $g$-part of the Kontsevich graph complex. 
\begin{Th} We have the canonical injection:
\begin{equation}\label{cgp}
 H^{\hdot}(\textsf {GC}_{2d}^2)\longrightarrow H^{\hdot+1}(\textsf{RGC}_d(\delta+\Delta_1))
\end{equation}
\end{Th}

\begin{proof} By Theorem \ref{CalC} we have an equivalence between the compactly supported cohomology of $\mathcal M_g$ and the cohomology of the Merkulov-Willwacher complex. For $g\geq 2$ by the Theorem of Chan-Galatius-Payne \cite{CGP1} (see also \cite{AWZ},\cite{AK3}) we have an inclusion:
$$
{H^{\hdot}(\textsf B_g\textsf {GC}_{2d}^{2})\cong W_0H_c^{\hdot+2dg}(\mathcal M_{g},\mathbb Q)\hookrightarrow H_c^{\hdot+2dg}(\mathcal M_g,\mathbb Q)}.
$$
Where $W_0$ is a weight zero piece of the weight filtration of the compactly supported cohomology of $\mathcal M_g$ \cite{DelH} \cite{DelH2}. For $g=1$ the result follows from the explicit computations of $\textsf B_1\textsf {GC}_{2d}^{2\,\hdot}(\delta)$ from \cite{Will} and the explicit computations of the cohomology of the hairy Getzler-Kapranov complex in genus one \cite{AK}.
\end{proof} 

\begin{remark}\label{Conj} Denote by $\textsf {OGC}_{2d+1}^{\hdot}(\delta)$ the \textit{oriented graph complex} \cite{Will2}. This complex naturally appears as being quasi-isomorphic to a deformation complex of the propered $\Lambda\textsf {LieB}_{d,d}$ \cite{MW2}. One can explicitly describe this complex as a complex with cochains given by pairs $(G,or_{G}),$ where $G$ is an oriented graph i.e. a graph equipped with a direction on each edge such that there are no closed cycles formed by paths \cite{Will2}. Applying the deformation complex to the morphism \eqref{kp} and passing to the cohomology, following \cite{MW} we get a morphism:
\begin{equation}\label{mw}
\textsf {mw}\colon H^{\hdot}(\textsf {OGC}_{2d+1})\longrightarrow H^{\hdot+1}(\textsf{RGC}_d(\delta+\Delta_1))
\end{equation} 
Note that it was shown in \cite{Will2} that the oriented graph complex $\textsf {OGC}_{2d+1}^{\hdot}(\delta)$ is quasi-isomorphic to the Kontsevich graph complex $\textsf {GC}_{2d}^{2\,\hdot}(\delta)$. Further, in \cite{Ziv} the explicit combinatorial quasi-isomorphism:
\begin{equation}\label{ziv}
\textsf {ziv}\colon \textsf {OGC}_{2d+1}^{\hdot}(\delta)\overset{\sim}{\longrightarrow} \textsf {GC}_{2d}^{2\,\hdot}(\delta)\end{equation} 
was constructed. We can organise this data into the diagram: 
\begin{equation}\label{conc1}
\begin{diagram}[height=3em,width=3em]
H^{\hdot}(\textsf {GC}_{2d}^2) & &  \rTo_{}^{\textsf{cgp}} &  &   \prod_{g=1}^{\infty}H_c^{\hdot+2dg}(\mathcal M_{g},\mathbb Q) &  \\
\uTo_{\sim}^{\textsf{ziv}} & & &  & \uTo_ {\sim}^{\textsf{cald}}  && \\
H^{\hdot}(\textsf {OGC}_{2d+1})& &  \rTo_{}^{\textsf{mw}} &  & H^{\hdot+1}(\textsf{RGC}_d(\delta+\Delta_1))   \\
\end{diagram}
\end{equation}
It was conjectured in \cite{AWZ} that the square above is commutative. This conjecture in particular implies the original conjecture from \cite{MW} that morphism \eqref{mw} is injective. From Theorem \ref{CalC} and the fact that $\nabla_1$ preserves the weight filtration (Corollary $4.4.4$ in \cite{AK}) the commutativity of square \eqref{conc1} follows from the commutativity of the following square:
\begin{equation}
\begin{diagram}[height=3em,width=3em]
H^{\hdot}(\textsf H_n\textsf {GC}_{2d}) & &  \rTo_{}^{\textsf{cgp}_n} &  &   \prod_{g\geq 0, 2g+n-2>0}^{\infty}H_c^{\hdot+2dg}(\mathcal M_{g,n},\mathbb Q) &  \\
\uTo_{\sim}^{\textsf{az}} & & &  & \uTo_ {\sim}^{\textsf{cost}}  && \\
H^{\hdot}(\textsf {H}_n\textsf {OGC}_{2d+1})& &  \rTo_{}^{\textsf{awz}} &  & H^{\hdot+n}(\textsf{RGC}_d^{[n]})   \\
\end{diagram}
\end{equation}
By $\textsf H_n\textsf {GC}_{2d}^{\hdot}(\delta)$ we have denoted the hairy graph complex \cite{CGP2} with labelled hairs, $\textsf {H}_n\textsf {OGC}_{2d+1}^{\hdot}(\delta)$ is the hairy oriented graph complex \cite{AWZ} and $\textsf{RGC}^{\hdot\, [n]}_d(\delta)$ is a ribbon graph complex which consists of ribbon graphs with $[n]$-labelled boundaries. Where $\textsf{az}$ is a version of morphism \eqref{ziv} for hairy graph complexes \cite{AWZ} (see also \cite{AZ}), $\textsf {cgp}_n$ is a "weight zero inclusion" \cite{CGP2} and $\textsf{awz}$ is a version of \eqref{mw} for hairy graph complexes.

\end{remark}

\begin{remark} We can define a degree $-1$ DG-Lie algebra structure on the totality of $C_c^{\hdot}(\mathcal M_g,\mathbb Q):$
$$
\{\bullet,\bullet\}\colon \prod_{g=2}^{\infty} C_c^{\hdot}(\mathcal M_g,\mathbb Q)\times \prod_{g=2}^{\infty} C_c^{\hdot}(\mathcal M_g,\mathbb Q)\longrightarrow \prod_{g=2}^{\infty} C_c^{\hdot+1}(\mathcal M_g,\mathbb Q)
$$
By the rule:
\begin{equation}\label{bracket}
 \{\bullet,\bullet\}:=\mathrm {Res}_{{\mathcal M}_{\Gamma} /\mathrm {Aut}(G)}\circ p_!\circ (\pi^*\otimes \pi^*)(\bullet,\bullet),\qquad G={\xy
 (0,0)*{{\bullet}^{g_1}}="a",
(6,0)*{{}^{g_2}{\bullet}}="b",
\ar @{-} "a";"b" <0pt>
\endxy}
\end{equation}
Where morphisms in the definitions are:
\begin{equation}\label{bracket1}
\begin{diagram}[height=2.5em,width=2.5em]
 \mathcal M_{G}:=\mathcal M_{g_1,1}\times \mathcal M_{g_2,1}  &   \rTo^{{p}}  &   {\mathcal M}_{G} /\mathrm {Aut}(G) &  \\
\dTo^{\pi\times \pi} & &  && \\
 \mathcal M_{g_1}\times \mathcal M_{g_2}   &      &    \\
\end{diagram}
\end{equation}
By $\pi$ we have denoted a forgetful (proper) morphism of stacks and $\mathrm {Res}_{{\mathcal M}_{\Gamma} /\mathrm {Aut}(\Gamma)}$ is a residue morphism along the normalised open stratum in $\overline{\mathcal M}_{g_1+g_2}.$ 
We believe that via Theorem \ref{CalC} shifted DG-Lie algebra structure \eqref{bracket} coincides with the combinatorial DG-Lie algebra structure on the Merkulov-Willwacher ribbon graph complex.

\end{remark}

\section{Appendix: The Teichmüller space of a bordered surface}
\subsection{Preliminaries} Let $(X,d)$ be a metric spaces. Recall that a \textit{geodesic} $\gamma$ joining points $x$ and $y$ is an isometric morphism $\gamma\colon [0,1]\rightarrow X$ such that $\gamma(0)=x$ and $\gamma(1)=y$, we write $[x,y]\subset X$ for the set $\gamma([0,1]),$ We say that the metric spaces is \textit{uniquely geodesic} if for any two points there exists a unique geodesic joining them. The important class of uniquely geodesic comes from \textit{$CAT(0)$-spaces.} \cite{BH}. Namely we say that a complete non-empty metric space $(X,d)$ is $CAT(0)$-space\footnote{Sometimes also called \textit{Hadamard spaces.}} if for any two points $x,y\in X$ there exists a midpoint $m$ for these two points i.e. $d(x,m)=d(y,m)=d(x,y)/2$ (see \cite{BH} Definition $1.1$).\footnote{Note that in \cite{BH} it is assumed that $CAT(0)$-spaces are not necessarily complete.} One can show that every $CAT(0)$-space is uniquely geodesic with geodesic varying continuously \cite{BH} (Proposition $1.4$). We say that a subset $C\subset X$ is \textit{convex} if for any $x$ and $y$ in $C$ the unique geodesic $[x,y]$ joining points $x$ and $y$ lies in $C.$ Recall that for a point $x\in X$ in a metric spaces $X$ and a subset $C$ in $X$ a distance $d(x,C)$ from the point $x$ to the subset $C$ is defined by the rule $d(x,C):=\inf_{y\in C} d(x,y).$ We have the following well-known result (Proposition $2.4$ \cite{BH}):

\begin{Prop}\label{ortp} Let $X$ be $CAT(0)$-space then for every closed convex subset $C$ and $x\in X$ there exists a unique point $\pi_C(x)\in X$ such that:
$$
d(x,C)=d(c,\pi_C(x))
$$
The point $\pi_C(x)$ is called the \textit{nearest point} in $C$ to a point $x.$ 
\end{Prop} 
Hence for any closed and convex subset in $CAT(0)$-space one can define the \textit{nearest point projection}:
$$
\Pi_C\colon X\longrightarrow C,\quad \Pi_C\colon x\mapsto \pi_C(x).
$$
This morphism is $1$-Lipschitz and hence continues. Moreover one can show (Proposition $2.4$ \cite{BH}) that the inclusion $C\longrightarrow X$ is a deformation retract and in particular taking $C$ to be a point one can show that every $CAT(0)$-space is contractible. 

\par\medskip 
Let $S_{g,n}$ be a closed topological surface of genus $g$ with $n$-punctures. We define the \textit{Teichmüller space} $\mathcal T_{g,n}$ of $S_{g,n}$ by the rule: a set of pairs $(X,g),$ where $X$ is a Riemann surface and $g\colon \Sigma\rightarrow X$ is a diffeomorphism modulo the equivalence: $(X,g) \cong (Y,f)$ if and only if $fg^{-1}$ is isotopic to a biholomorphic morphism. One can show that the Teichmüller space $\mathcal T_{g,n}$ is a manifold with the cotangent space defined by the rule: a fiber over a point $(X,g)$ is given by the space $Q(X)$ of holomorphic quadratic differentials on $X$ with at most simple poles at the punctures. Recall that every Riemann surface $X$ has a hyperbolic structure with a hyperbolic metric $dh^2.$ We define the Petersson Hermetian pairing i.e. a structure of the Riemannian manifold on $\mathcal T_{g,n}$ by the rule:
$$
<\psi,\phi>:=\int_{S_{g,n}} \phi \overline{\psi}(dh^2)^{-1}
$$
Then we define the \textit{Weil-Petersson metric} $d_{WP}(x,y)$ on $\mathcal T_{g,n}$ as the dual to the Hermitian pairing above \cite{AHL}. The important fact is that this metric is Kähler and uniquely geodesic. 

Denote by $C(S_{g,n})$ the complex of curves in $S_{g,n}.$ This is a simplicial complex with $0$-simplicies given by homotopy classes of non-trivial nonperipheral simple closed curves in $S_{g,n}.$ A maximal simplex has $3g-3+n$ element. A free homotopy class $\alpha$ of a curve in $S_{g,n}$ determines a length function $\ell_{\alpha}$ on a Teichmüller space $\mathcal T_{g,n}$ namely on an element $(X,f)$ we set a value of $\ell_{\alpha}$ to be a length of a geodesic homotopic to $f(\alpha).$ For a simplex $\sigma\in C(S_{g,n})$ we define a space $\mathcal T(\sigma):=\{X\,|\, \ell_{\alpha}(X)=0\, \mathrm{iff}\, \alpha \in \sigma\}.$ The completion of the Teichmüller space with respect to the Weil-Petersson metric $d_{WP}$ will called the \textit{augmented Teichmüller space} and denoted by $\overline{\mathcal T}_{g,n}$ \cite{ABI}. We have the following \cite{Mas}:
\begin{Th} The boundary $\partial \overline{\mathcal T}_{g,n}:=\overline{\mathcal T}_{g,n}\setminus \mathcal T_{g,n}$ of the Weil-Petersson completion of the Teichmüller space $\mathcal T_{g,n}$ is given by:
$$
\partial\overline{\mathcal T}_{g,n}\cong \bigcup_{\sigma \in C(S_{g,n})} \mathcal T(\sigma)
$$
\end{Th} 
The important result that  the augmented Teichmüller space $\overline{\mathcal T}_{g,n}$ is $CAT(0)$-space \cite{Yam}. Note that a closure $\overline{\mathcal T(\sigma)}:=\cup_{\gamma\subset \sigma} \mathcal T (\gamma)$  of a stratum $\mathcal T(\sigma)$ is a convex closed subspace $i_{\sigma}\colon \overline{\mathcal T(\sigma)}\longrightarrow \overline{\mathcal T}_{g,n}$ and hence the nearest point projection (Lemma \ref{ortp}) is defined:
\begin{equation}\label{ortp2}
\Pi_{\overline{\mathcal T_{\sigma}}}\colon \overline{\mathcal T}_{g,n}\longrightarrow \overline{\mathcal T(\sigma)}
\end{equation}
Note that this morphism is left dual to the canonical closed inclusion:
\begin{equation}\label{inct1}
i_{\sigma}\colon  \overline{\mathcal T(\sigma)}\hookrightarrow \overline{\mathcal T}_{g,n}
\end{equation}
\par\medskip 
Let $Diff^+(S_{g,n})$ be a group of diffeomorphisms of the topological surface. By $Diff^0(S_{g,n})$ we denote the connected component of identity i.e. diffeomorphisms isotopic to identity. We define the \textit{mapping class group $MCG_{g,n}$ of a surface of genus $g$ with $n$ marked points} by the rule:
$$
MCG_{g,n}:=Diff^+(S_{g,n})/Diff^0(S_{g,n})
$$
The mapping class group acts on the Teichmüller space by the rule: $g(X,f):=(X,fg^{-1}).$ This action is properly discontinuous. We have an equivalence of orbifolds:
\begin{equation}\label{qt1}
\mathcal M_{g,n}\cong [\mathcal T_{g,n}/ MCG_{g,n}].
\end{equation}
Note that the action of the mapping class group extends to the augmented Teichmüller space through it is not properly discontinues. Following \cite{HARV} we have :
\begin{Th}\label{qt2} For any $g,$ and $n$ such that $2g+n-2>0$ we have an equivalence of topological stacks:
$$
\overline{\mathcal M}_{g,n}\cong [\overline{\mathcal T}_{g,n}/ MCG_{g,n}].
$$
\end{Th}
Let $\sigma \in C(S_{g,n})$ we denote by $MCG_{g,n}^{\sigma}$ the subgroup of $MCG_{g,n}$ consisting of those elements which have representative $g\colon S_{g,n}\longrightarrow S_{g,n},$ such that $g(\gamma)$ is homotopic to $\gamma$
for any $\gamma\in \sigma$ and which fixes each component of $S_{g,n}\setminus \sigma.$ It is easy to see that a morphism \eqref{inct1} is $(MCG_{g,n}^{\sigma},MCG_{g,n})$-equivariant and hence we have the morphism between the corresponding quotients: 
\begin{equation}\label{inct2}
i_{\sigma}\colon [\overline{\mathcal T(\sigma)}/ MCG_{g,n}^{\sigma}]\cong \mathcal D_{\Gamma(\sigma)}\hookrightarrow  \overline{\mathcal M}_{g,n}
\end{equation}
Here $\mathcal D_{\Gamma}$ is a component of a boundary of the Deligne-Mumford moduli stack $\overline{\mathcal M}_{g,n}$ i.e. closure of a locus of stable curves with a dual graph being $\Gamma(\sigma)$ (this graph depends on $\sigma$ in a usual way).
\subsection{Teichmüller space ${}^B\mathcal T_{g,n}$} 

Let $\Sigma$ be a topological surface of genus $g$  with $n$-boundary components. By $D\Sigma_{g,n}:=\Sigma_{g,n}\cup_{\partial \Sigma_{g,n}} \overline{\Sigma}_{g,n}$ we denote the double of this surface i.e. the closed topological surface of genus $2g+n-1$ defined by gluing the surface with the oppositely oriented one along the boundary. We have an involutive diffeomorphism $\tau\colon D\Sigma_{g,n}\longrightarrow D\Sigma_{g,n}$ which act by interchanging $\Sigma_{g,n}$ and $\overline{\Sigma}_{g,n}.$ An element $\tau$ generates a subgroup in a mapping class group of $D\Sigma_{g,n}$ which by abuse of notation we will denote by the symbol $\tau\subset MCG_{2g+n-1}.$ As an element of the mapping class group $\tau$ acts on the Teichmüller space $\mathcal T_{2g+n-1}:$
\begin{align}\label{ra}
&\tau\colon \mathcal T_{2g+n-1}\longrightarrow  \mathcal T_{2g+n-1}\\
& \tau\colon (X,f)\longmapsto (X,f\tau^{-1}).
\end{align}
We consider the  $\tau$-invariant subspace $\mathcal T_{2g+n-1}^{\tau}$ in the Teichmüller space of $D\Sigma_{g,n}.$ By the definition elements in this space are pairs $(X,f)$ where $X$ is a Riemann surface of genus $2g+n-1$ and $f\colon D\Sigma_{g,n}\rightarrow X$ is a diffeomorphism such that $f\tau f^{-1}\colon X\longrightarrow X$ is isotopic to a biholomorphic morphism, modulo the usual equivalence. We denote the locus of fixed points under the diffeomorphism $f\tau f^{-1}$ by $X^{\mathbb R}.$ 

\begin{Def}\label{tb1} For a bordered topological surface $\Sigma_{g,n}$ we define the \textit{Teichmüller space ${}^B\mathcal T_{g,n}$ of $\Sigma_{g,n}$} as elements $(X,f)\in \mathcal T_{2g+n-1}^{\tau}$ with a choice of a component in
$X\setminus X^{\mathbb R}.$ 
\end{Def} 
By the construction we have a finite morphism ${}^B\mathcal T_{g,n}\longrightarrow \mathcal T_{2g+n-1}^{\tau}$ via this morphism we equip the Teichmüller space ${}^B\mathcal T_{g,n}$ with a topology. Let $MCG_{2g+n-1}$ be a mapping class group of a topological surface $D\Sigma_{g,n}.$ We give:
\begin{Def}\label{tb2} For a bordered surface $\Sigma_{g,n}$ we define the \textit{mapping class group $MCG^B_{g,n}$ of $\Sigma_{g,n}$} as a centralizer of $\tau$ in the mapping class group of $D\Sigma_{g,n}:$
$$
MCG_{g,n}^B:=C_{MCG_{2g+n-1}}(\tau)
$$

\end{Def} 
Note that  $\mathcal T_{2g+n-1}^{\tau}$ is stable under the action of the mapping class group. Let $(X,f)\in \mathcal T_{2g+n-1}^{\tau}$ and $g\in MCG_{g,n}^B$ then we have $(X,fg^{-1})$ and $fg^{-1}\tau gf^{-1}.$ Since $\tau g$ is isotopic to $g\tau$ we get that $fg^{-1}\tau gf^{-1}$ is isotopic to a biholomorphic morphism. We lift the action of the mapping class group to ${}^B\mathcal T_{g,n}$ by the obvious rule. We have the following:

\begin{Prop}\label{tb3} For any $g,n$ such that $2g+n-2>0$ we have an equivalence of orbifolds:
$$
\mathcal N_{g,n}\cong [{}^B\mathcal T_{g,n}/ MCG_{g,n}^B].
$$

\end{Prop}

\begin{proof} By the construction elements in $[{}^B\mathcal T_{g,n}/ MCG_{g,n}^B]$ are triples $(X,f,C)$ such that $f\tau f^{-1}\colon X\overset{\sim}\longrightarrow X$ is an involutive biholomorphic morphism i.e. a real structure on $X$ and $C$ is a component of $X\setminus X^{\mathbb R}.$ This datum is one to one corresponds to a bordered surface of genus $g$ with $n$ boundary components \cite{Cost}. The rest follows from the construction of topology on $\mathcal N_{g,n}$ \cite{Liu} (cf. \eqref{qt1}).
\end{proof} 

Note that the action \eqref{ra} extends to the augmented Teichmuller space and we consider the corresponding fixed point locus $\overline{\mathcal T}_{2g+n-1}^{\tau}.$ For an element $(X,f)\in\overline{\mathcal T}_{2g+n-1}^{\tau}$ we consider the normalization of every singularity in $X^{\mathbb R},$ which will be a union of connected components $X_{i}.$ Hence we give:
\begin{Def}\label{tb4} For a bordered topological surface $\Sigma_{g,n}$ we define the \textit{augmented Teichmüller space ${}^B\overline{\mathcal T}_{g,n}$ of $\Sigma_{g,n}$} as an element $(X,f)\in \overline{\mathcal T}_{2g+n-1}^{\tau}$ with a choice of a component in $X_i\setminus X^{\mathbb R}_i$ for each component $X_i.$ 
\end{Def} 
Via a forgetful morphism ${}^B\overline{\mathcal T}_{g,n}\longrightarrow  \overline{\mathcal T}_{2g+n-1}^{\tau}$ we equip the augmented Teichmüller space of a bordered surface with a topology. We have the following:

\begin{Prop}\label{tb5} We have an equivalence of topological stacks:

$$
[{}^B\overline{\mathcal T}_{g,n}/MCG_{g,n}^B]\cong \overline{\mathcal N}_{g,n}^L
$$

\end{Prop}

\begin{proof} The result follows from \cite{HARV} and the construction of Liu's compactification \cite{Liu}.

\end{proof}

Note that the space $\overline{\mathcal T}_{2g+n-1}^{\tau}$ is $CAT(0)$-space since it is given as a fixed point set of the group acting by isometries. The augmented Teichmüller space  ${}^B\overline{\mathcal T}_{g,n}$ of a bordered surface inherits properties of this metric. In particular $\sigma\in C(D\Sigma_{g,n})$ be a simplex which is invariant under $\tau$ i.e. $\tau(\gamma)$ is homotopic to $\gamma$ for any $\gamma\in \sigma.$ Analogically to the definition of the stratum in the Teichmüller space we define the stratum ${}^B\mathcal T(\sigma).$ Hence we define the nearest point projection:
\begin{equation}\label{ortp3}
\Pi_{{}^B\overline{\mathcal T(\sigma)}}\colon {}^B\overline{\mathcal T}_{g,n}\longrightarrow {}^B\overline{\mathcal T(\sigma)}
\end{equation} 
by the rule $\Pi_{{}^B\overline{\mathcal T(\sigma)}}(X,f,(X_i)):=(\Pi_{{}^B\overline{\mathcal T(\sigma)}}(X,f),(\Pi_{{}^B\overline{\mathcal T(\sigma)}}(X)_i).$

\bibliographystyle{amsalpha}
\bibliography{tt}

\end{document}